\documentclass[12pt, a4paper, oneside]{article}
\usepackage{amsmath, amsthm, amssymb, bm, graphicx,  mathrsfs,geometry,color,cite}
\usepackage[title]{appendix} 
\usepackage{lineno} 
\usepackage{booktabs} 
\usepackage{array} 
\usepackage{caption}
\usepackage{subfigure}
\usepackage{titling} 
\usepackage{authblk}
\usepackage{etoolbox}

\usepackage[hyperfootnotes=false]{hyperref}

\newtheorem{theorem}{Theorem}[section]
\numberwithin{equation}{section}

\newtheorem{lemma}[theorem]{Lemma}

\newtheorem{proposition}[theorem]{Proposition}
\newtheorem{remark}[theorem]{Remark}

\usepackage{titlesec}
\titleformat{\section}
{\fontsize{14}{16}\selectfont\bfseries}
{\thesection}{1em}{}

\pretitle{\begin{center}\fontsize{16}{18}\selectfont\bfseries}      
	\posttitle{\end{center}\vskip 1em}           
\preauthor{\begin{center}\large}             
	\postauthor{\end{center}}                    
\predate{\begin{center}\small}               
	\postdate{\end{center}}                      
\date{}

\title{{Variational characterizations of weighted eigenvalue problems for nonlocal dispersal operators with application to the basic reproduction ratio}}

\author[a]{Xiandong Lin}
\author[b]{Jiazhuo Cheng}
\author[a]{Qiru Wang \thanks{Corresponding Author: Email: mcswqr@mail.sysu.edu.cn}}

\affil[a]{School of Mathematics\,(Zhuhai), Sun Yat-sen University, Zhuhai 519082, Guangdong, P.R. China}
\affil[b]{School of Mathematics and Systems Science, Guangdong Polytechnic Normal University, Guangzhou 510665, Guangdong, P.R. China}

\date{}
\usepackage{sectsty}
\sectionfont{\fontsize{14}{16}\selectfont} 
\begin{document}
	
	\maketitle
	
	\vspace{-3em}
	
\begin{abstract}
	Nonlocal dispersal operators, typically modeled via integral operators, describe long-range dispersal mechanisms. The variational characterization of weighted eigenvalue problems for such operators is challenging because the principal eigenvalue may fail to exist due to non-compactness. In this paper, we establish  variational characterizations of weighted eigenvalue problems even in the absence of a principal eigenvalue. To this end, we first study the spectral bound of a class of such operators, obtaining a Collatz--Wielandt characterization and, in the self-adjoint case, a Rayleigh--Ritz characterization.  By introducing a parameter and using the equivalence of these characterizations of the spectral bound, we show that the different definitions of the weighted eigenvalue problem coincide, thereby yielding variational characterizations of weighted eigenvalue problems. As an application, we derive explicit variational characterizations of the basic reproduction ratio and apply them to an SIS epidemic model with saturated incidence to study the limiting behavior of the basic reproduction ratio.  In particular, we find that when the saturation coefficient is degenerate, the limiting behavior of the basic reproduction ratio as the total population tends to zero differs strikingly from that in local diffusion.
\end{abstract}
		


\section{Introduction}

Nonlocal dispersal, typically modeled via integral operators, has been increasingly employed to capture long-range dispersal mechanisms, offering an alternative to traditional local diffusion models. In the classical reaction--diffusion setting, the principal eigenvalue plays a central role: it characterizes the stability of equilibrium and provides persistence criteria \cite{MR2214420,MR1014659,MR1100011}. However, the solution operators of nonlocal dispersal models are generally non-compact, and consequently a principal eigenvalue may fail to exist. This lack of compactness introduces significant challenges for the study of eigenvalue problems and has attracted substantial research interest in recent years.

To overcome this difficulty, various analytical approaches have been developed. For instance, Shen and collaborators \cite{MR3285842, MR3637938, MR3000610}, and Kang and Ruan \cite{MR4476225,MR4616638} applied perturbation theory for positive semigroups, as introduced by B\"urger \cite{MR923493}. In a different line of inquiry, Coville \cite{MR2718672, MR3548277}, Zhang \cite{MR4803717, MR3906242, MR5005572}, and Li \cite{MR4929554, MR4754238, MR4104463, MR3632209},  and coauthors, employed the generalized Krein--Rutman theorem. Meanwhile, Wang and his group \cite{MR3583503, MR4601060, MR4612706} developed a Touching lemma to investigate eigenvalue problems for nonlocal dispersal operators.

The above research on principal eigenvalues of nonlocal operators can be broadly divided into two categories. The first adopts the notions developed for characterizing the spectral properties of elliptic operators (see, e.g., \cite{MR1258192,MR361998,MR0425380}); in this direction, the authors in \cite{MR3548277,MR5005572,MR4601060}, among others, introduced the generalized principal eigenvalue through \(\bar\lambda_p\), \(\underline\lambda_p\), and \(\lambda_v\) (see Section 2 for the definitions). For a single equation, Berestycki et al. \cite{MR1258192} proved that these notions coincide. The second category is based on the spectral structure of the operator: the authors in \cite{MR3285842,MR3637938,MR4628895}, among others, defined the principal spectrum point as the principal eigenvalue of the operator and studied its properties.

When a principal eigenvalue exists, the different definitions yield the same value, which is obvious for second-order uniformly elliptic operators and also holds for nonlocal dispersal operators. The challenge lies precisely in the case where the principal eigenvalue fails to exist.  In Section 2, we prove that these definitions still coincide.
To this end, building on the studies of principal eigenvalues for nonlocal dispersal operators, particularly the approximation techniques introduced in \cite[Lemma 4.2]{MR4929554} and \cite[Theorem A]{MR5005572}, we demonstrate that the generalized eigenvalue coincides with the spectral bound, and thereby establish a Collatz--Wielandt characterization of the spectral bound under a relaxed irreducibility condition, 
clarifying the nature of generalized eigenvalues. Moreover, in the self-adjoint case, we provide a Rayleigh--Ritz characterization of the spectral bound.

The problem associated with the principal eigenvalue is the weighted eigenvalue problem. For second-order uniformly elliptic operators, the characterization of the weighted eigenvalue has been extensively studied in \cite{MR0190527,MR0466940,MR0588690,MR0305150,MR1878528} and the references therein. For the nonlocal dispersal problem, if the associated weighted eigenvalue problem admits a principal eigenvalue, one can adopt methods developed for second-order uniformly elliptic operators to characterize it. For example, Yang et al. \cite{MR3945624} established a Rayleigh--Ritz characterization under this assumption for a single equation.

However, a fundamental challenge in this setting is that the principal eigenvalue may not exist for nonlocal dispersal operators. This absence causes technical difficulties, as the classical variational characterizations of the weighted eigenvalue problem rely on the existence of a principal eigenvalue. While Nussbaum and Pinchover \cite{MR1226957} developed variational principles for the generalized principal eigenvalue of second-order elliptic operators on general domains, where a principal eigenvalue may fail to exist, their approach relies on local elliptic theory and thus does not extend to nonlocal dispersal operators. 
In the nonlocal setting, a principal eigenvalue may again fail to exist, and new arguments are required to determine whether the different definitions of the weighted eigenvalue coincide and how to characterize it variationally.

In Section 3, we address these issues by   introducing a parameter $\mu>0$ and a family of nonlocal dispersal operators on functions \(\phi: \overline{\Omega} \to \mathbb{R}^{m}\):
\[
[\mathcal{L}_{\mu} \phi](x) := D[\mathcal{J}\phi](x) + A(x)\phi(x) + \frac{1}{\mu} F(x)\phi(x), \quad x\in \overline{\Omega},
\]
where  \(\Omega \subset \mathbb{R}^n\) is a smooth bounded domain, \(D:=\mathrm{diag}\{d_{1},d_{2},\dots,d_{m}\}\), \([\mathcal{J}\phi](x) :=(\int_{\Omega}J_{1}(x-y)\phi_{1}(y)\,\mathrm{d}y, \dots, \int_{\Omega}J_{m}(x-y)\phi_{m}(y)\,\mathrm{d}y)^{T}\), \(A(x) = (a_{ij}(x))_{m \times m}\) is a cooperative matrix, and \(F(x) = (f_{ij}(x))_{m \times m}\) is a nonnegative matrix with \(\max_{x\in\overline{\Omega}}s(F(x))>0\), where \(a_{ij}, f_{ij}\in C(\overline{\Omega}, \mathbb{R})\).

In accordance with the two types of definitions of the principal eigenvalue, the weighted eigenvalue can be defined analogously. In this work, we adopt the spectral-bound definition: let \(\mu_0\) be the unique positive solution of \(s(\mathcal L_\mu)=0\), where \(s(\mathcal L_\mu)\) denotes the spectral bound of \(\mathcal L_\mu\). Equivalently, \(\mu_0\) is the reciprocal of the classical weighted eigenvalue \(\lambda_0\) of the problem \(-(D\mathcal J+A)\phi=\lambda F\phi\), i.e., \(\mu_0=\frac{1}{\lambda_0}\). In Section 3, we focus on \(\mu_0\) since it coincides with the basic reproduction ratio, as will be shown in Section 4. Building on the characterizations of the spectral bound established in Section 2 and the properties of \(s(\mathcal{L}_{\mu})\) with respect to \(\mu\), we show that \(\mu_0\) coincides with the other type of definition. Hence, we establish a Rayleigh--Ritz characterization even in the absence of a principal eigenvalue when the operator is self-adjoint. Moreover, we establish a Collatz--Wielandt characterization for the general case (see Theorems \ref{2.1} and \ref{2.3}).

As an application of the weighted eigenvalue problem, we consider the basic reproduction ratio, which is a crucial threshold parameter in infectious disease models and other biological mathematical frameworks.
The theory developed  by Zhang and Zhao \cite{MR4348165} provides a robust approach to connecting the basic reproduction ratio with associated principal eigenvalue problems. In particular, this framework enables us to characterize the limiting behavior of the basic reproduction ratio in terms of the limiting profiles of the principal eigenvalue (see, e.g., \cite{MR4628895, Feng_Li_Yang_2024, MR3945624, MR4553723, MR4620151}).

According to the theory of the basic reproduction ratio \cite{MR3612966, MR3992071, MR4348165, MR2505085}, under suitable hypotheses, the basic reproduction ratio is the unique solution to $s(\mathcal{L}_{\mu}) = 0$.
In Section 4, we apply the characterizations for weighted eigenvalue problems  to derive variational characterizations of the basic reproduction ratio, which  yields an explicit expression for the basic reproduction ratio and offers a powerful tool for analyzing its properties.
In Subsection 4.2, we further consider an SIS model with saturated incidence  and use the established variational characterizations of the basic reproduction ratio to study its limiting behavior, which provides an alternative approach to that of Zhang and Zhao \cite{MR4348165}.
In particular, when the saturation coefficient is degenerate, i.e., it vanishes at some points, the limiting profiles of the basic reproduction ratio differ markedly from those in local diffusion models as the total population tends to zero.

The remainder of this paper is organized as follows. In Section 2, we study the spectral bound for nonlocal dispersal operators and establish its Collatz--Wielandt and Rayleigh--Ritz characterizations. In Section 3, we establish variational characterizations of $\mu_{0}$, covering both the non-degenerate and partially degenerate cases. In Section 4,  we apply these results to obtain characterizations of the basic reproduction ratio and use them to analyze its limiting behavior for an SIS epidemic model.

\section{The spectral bound}

In this section, we investigate variational characterizations of the spectral bound for $L$:
\[
\left[ L\phi\right](x):=D \left[ \mathcal{J}\phi\right](x) + M(x)\phi(x), \quad  x\in \overline{\Omega},
\]
where $M(x)=\left( m_{ij}(x)\right)_{m\times m}$ is a  cooperative matrix with $m_{ij}\in  C(\overline\Omega, \mathbb{R})$.
These characterizations serve as the foundation for the subsequent study of the weighted eigenvalue problem in Section 3. We show that the spectral bound coincides with the generalized eigenvalue, which yields the Collatz--Wielandt characterization. Furthermore, for the self-adjoint case, we obtain a Rayleigh–Ritz characterization of the spectral bound.
Let   $\mathcal{S}:=\left\lbrace 1,2,\dots,m\right\rbrace $. We make the following assumptions on $D$ and $J_{i}$:

\begin{itemize}
	\item [(\textbf{J})] $J_{i}$ is a continuous nonnegative function with $J_{i}(0) > 0$, $\int_{\mathbb{R}^{n}} J_{i}(x)\, \mathrm{d}x = 1$, $J_{i} \in L^{\infty}(\mathbb{R}^{n})$, for all $i\in \mathcal{S}$;
	
	\item[(\textbf{D})] $d_{i}>0$ for all $i\in \mathcal{S}$.
\end{itemize}
In this section, we  always assume that (\textbf{J}) and (\textbf{D}) hold.

Define the norm on $ \mathbb{R}^{m}$ as $\left\| a\right\|_{\mathbb{R}^{m}}:= \sqrt{a^{T}a}$,
where $a=(a_{1},\dots,a_{m})^{T}\in\mathbb{R}^{m} $. We denote $X:=C(\overline\Omega,\mathbb{R}^{m})$,
equipped with the maximum norm
$\left\| \varphi\right\| _{X}:=\max\limits_{x\in\overline\Omega } \left\| \varphi(x)\right\|_{\mathbb{R}^{m}}$,
and the positive cone $X_{+}:=C(\overline\Omega,\mathbb{R}^{m}_{+})$.
Let $a,b\in\mathbb{R}^{m} $. We write $a-b\ge 0$ if $a-b\in\mathbb{R}^{m}_{+}$; $a-b> 0$ if
$a-b\in\mathbb{R}^{m}_{+}\setminus\left\{0\right\}$; and $a-b\gg 0$ if $a-b\in\mathrm{Int}(\mathbb{R}^{m}_{+})$. In particular,  we say $a$ is positive if $a\ge 0$, strictly positive if $a>0$, and strongly positive if $a\gg 0$.
We also denote
$L^{2}:=L^{2}({\Omega}, \mathbb{R}^{m})$ equipped with the norm
$ \left\| \phi\right\|_{L^{2}}:= \left( \int_{\Omega}\phi^{T}(x)\phi(x) \, \mathrm{d}x\right)^{\frac{1}{2}}$, and the positive cone $L^{2}_{+}:=\left\lbrace \phi\in L^{2}: \phi(x) \ge 0 \text{ for almost every } x \in \overline\Omega  \right\rbrace $.
Similar to the order defined in $\mathbb{R}^{m}$, we can define the partial order induced by the positive cone in $X$ and $L^{2}$, respectively.

Note that
the operator $L$ can act on either $X$ or $L^2$. We write $L_C$ for the former and $L_{L^2}$ for the latter to distinguish them. Subscripts may be omitted when the domain is clear from context or both interpretations hold.

Let $\sigma(L)$ denote the spectrum of $L$, and let $s(L)$ be its spectral bound, defined by
\[
s(L) := \sup \{\operatorname{Re} \mu : \mu \in \sigma(L)\}.
\]
We denote by $\mathcal{R}(L)$ and $\mathcal{N}(L)$ the range and null space of $L$, respectively. The operator $L$ is called a Fredholm operator if $\mathcal{R}(L)$ is closed and both
\[
\dim \mathcal{N}(L) < \infty \quad \text{and} \quad \operatorname{codim} \mathcal{R}(L) := \dim \big( E / \mathcal{R}(L) \big) < \infty,
\]
where $E$ denotes the space on which $L$ acts, i.e., $X$ or $L^2$ as appropriate.
Its Fredholm index is then defined as
\[
\operatorname{ind}(L) = \dim \mathcal{N}(L) - \operatorname{codim} \mathcal{R}(L).
\]
Following \cite[Section 7.5]{MR1861991}, the essential spectrum of $L$ is defined as
\[
\sigma_e(L) := \bigl\{ \lambda \in \sigma(L) : \lambda I - L \text{ is not a Fredholm operator of index zero} \bigr\}.
\]
The spectral radius and the essential spectral radius of $L$ are denoted by $r(L)$ and $r_e(L)$, respectively. Since Nussbaum \cite{MR264434} shows that various definitions of the essential spectral radius are equivalent, we may adopt the following one:
\[
{r_e}\left( L \right):= \sup \left\{ {\left| \lambda  \right|:\lambda  \in {\sigma _e}\left( L \right)} \right\}.
\]
$L$ is said to admit a principal eigenvalue if $s(L)$ is an eigenvalue whose associated eigenfunction is strictly  positive.
We introduce an assumption on $M$:
\begin{itemize}
	\item [(\textbf{M})] $M$ is weakly irreducible (fully coupled), in the sense that there do not exist two disjoint nonempty subsets  $\mathbb{I}$ and $\mathbb{J}$, with $\mathbb{I}\cup\mathbb{J}= \mathcal{S} $, such that $ m_{ij}(x) \equiv0$ in $\overline{\Omega}$ for all $i\in \mathbb{I}, j\in \mathbb{J}$.
\end{itemize}
In \cite{MR4612706,MR3583503,MR2718672,MR5005572,MR4601060,MR4929554,MR3548277} and references therein, the generalized principal eigenvalue of  $ L_{C}$ is defined by
\[
\begin{aligned}
	\bar \lambda_{p}: =& \inf \left\lbrace \lambda\in \mathbb{R}: \exists \phi \in \mathrm{Int}(X_{+}) \ s.t. \ \left[ L\phi\right](x)\le \lambda \phi(x), \forall x\in \overline{\Omega}  \right\rbrace\\[0.5em]
	=& \inf_{\phi \in \mathrm{Int}(X_{+})} \sup_{x \in \Omega, i \in \mathcal{S}} \frac{d_i \int_{\Omega} J_i(x- y) \phi_i(y) \,\mathrm{d}y + \sum_{j=1}^m m_{ij}(x)\phi_{j}(x) }{\phi_i(x)},
\end{aligned}
\]
or
\[
\begin{aligned}
	\underline \lambda_{p}:= & \sup \left\lbrace \lambda\in \mathbb{R}: \exists \phi \in \mathrm{Int}(X_{+}) \ s.t. \ \left[ L\phi\right](x)\ge \lambda \phi(x), \forall x\in \overline{\Omega}   \right\rbrace\\[0.5em]
	=& \sup_{\phi \in  \mathrm{Int}(X_{+})} \inf_{x \in \Omega, i \in \mathcal{S}} \frac{d_i \int_{\Omega} J_i(x- y) \phi_i(y) \,\mathrm{d}y + \sum_{j=1}^m m_{ij}(x)\phi_{j}(x) }{\phi_i(x)}.
\end{aligned}
\]
The latter expressions for $\bar{\lambda}_p$ and $\underline{\lambda}_p$ are known as the Collatz--Wielandt characterization.

We note that the equality $\bar{\lambda}_p = \underline{\lambda}_p$ has been established for cases where $M$ is strongly irreducible (i.e. $M(x)$ is irreducible for all $x$)  \cite{MR4929554,MR4601060}, thereby showing that the two definitions are equivalent. Yet the intrinsic nature of the generalized eigenvalue remains unclear.
In this section, we  show that $\bar{\lambda}_p = \underline{\lambda}_p = s(L_{C}) = s(L_{L^{2}})$ under assumption (\textbf{M}), which will allow us to clarify the intrinsic nature of the generalized eigenvalue.

We point out that assumption (\textbf{M}) ensures that when $L$ admits a principal eigenvalue, the corresponding eigenfunction is strongly positive (see \cite[Theorem 2.2]{MR4628895}), thereby guaranteeing $\bar \lambda_{p} \le s(L) \le \underline \lambda_{p}$ holds in the presence of a principal eigenvalue.

When assumption (\textbf{M}) does not hold, even if $L$ admits a principal eigenvalue, the associated eigenfunction is not strongly positive. In defining $\bar{\lambda}_p$ and $\underline{\lambda}_p$, the most straightforward adjustment is to enlarge the admissible set of test functions $\phi$ to include all strictly positive functions, rather than only strongly positive ones. However, in this case, the formulas can no longer be written in the latter form---the Collatz--Wielandt characterization. Moreover, the equality between $\bar{\lambda}_p$ and $\underline{\lambda}_p$ remains uncertain. This can be seen in the following example.

\noindent\textbf{Example:}
Consider $m=2$, $d_1=d_2=1$, $J_1=J_2$, and
$$
M(x)=\begin{pmatrix}2&0\\0&1\end{pmatrix}.
$$
Direct calculation gives $s(L_C) = \bar{\lambda}_p = \underline{\lambda}_p + 1$. However, if the test functions $\phi$ are enlarged to all strictly positive (not necessarily strongly positive) functions, then  $\bar{\lambda}_p + 1 = \underline{\lambda}_p = s(L_C)$.

Let $\mathcal{M}_C: X \to X$ and $\mathcal{M}_{L^2}: L^2 \to L^2$ be the multiplication operators given by $(\mathcal{M}_C\phi)(x)=M(x)\phi(x)$ for $\phi\in X$, and analogously for $\mathcal{M}_{L^2}$ with $\phi\in L^2$.

\begin{proposition}\label{pro2.1}
	The following statements hold:
	\begin{itemize}
		\item [\rm(i)] $\sigma_e(\mathcal{M}_{C}) = \sigma(\mathcal{M}_{C}) = \bigcup_{x \in \overline\Omega} \sigma(M(x))$;
		
		\item [\rm(ii)] $\sigma_e(\mathcal{M}_{L^{2}}) = \sigma(\mathcal{M}_{L^{2}}) = \bigcup_{x \in \overline\Omega} \sigma(M(x))$.
	\end{itemize}
\end{proposition}

\begin{proof}
	The statement (i) is proved in \cite[Proposition 2.7]{MR3906242}. By a similar discussion as in \cite[Proposition 2.7]{MR3906242}, we can prove statement (ii).
\end{proof}

Note that $D \mathcal{J}$ is a compact operator on $X$, and also on  $L^{2}(\overline{\Omega},\mathbb{R}^{m})$. By \cite[Theorem 7.26]{MR1861991}, we have $\sigma_{e}(L_{C})=\sigma_e(\mathcal{M}_{C})$ and $\sigma_{e}(L_{L^{2}})=  \sigma_e(\mathcal{M}_{L^{2}}) $. By Proposition \ref{pro2.1}, we obtain that $\sigma_{e}(L_{C})=\sigma_e(L_{L^{2}})$.

We can choose $c$ sufficiently large such that $cI+M(x)$ is a nonnegative  matrix, and then $cI +L$ is a positive linear operator. According to \cite[p.276]{MR1921782}, we have $s(cI +L) =r( cI +L)$. Clearly, $s(cI +L)= c + s(L) $,   and $r_{e}(cI +L)= c+ \max\limits_{x \in \overline{\Omega}}s(M(x))$. According to  Krein--Rutman theorem (see \cite[Corollary 2.2]{MR643014}), $cI +L$  admits the principal eigenvalue if $r( cI +L)>r_{e}(cI +L) $. Hence, we have the following result.

\begin{proposition}\label{pro2.2}
	$L$  admits the principal eigenvalue if $s(L)> \max\limits_{x \in \overline{\Omega}}s(M(x)) $.
\end{proposition}

\begin{lemma}\label{l2.1}
	$ s(L_{C})=s(L_{L^{2}})$.
\end{lemma}

\begin{proof}
	We prove the equality by considering two cases separately.

	\textbf{Case 1.} $s(L_{L^{2}}) =  \max\limits_{x \in \overline{\Omega}} s(M(x))$.
	
	We show that $s(L_{C}) =  \max\limits_{x \in \overline{\Omega}} s(M(x))$. Suppose, for contradiction, that $s(L_{C}) >  \max\limits_{x \in \overline{\Omega}} s(M(x))$.
	By Proposition \ref{pro2.2}, there exists $\phi \in X_{+}$ such that $L_{C} \phi = s(L_{C}) \phi$. This implies $s(L_{L^{2}}) \ge s(L_{C})$, which contradicts the assumption.

	\textbf{Case 2.} $s(L_{L^{2}}) >  \max\limits_{x \in \overline{\Omega}} s(M(x))$.
	
	By Proposition~\ref{pro2.2}, there exists $\phi \in L^{2}_{+}$ such that $L_{L^{2}} \phi = s(L_{L^{2}}) \phi$. Since $D\mathcal{J}(L^{2})\subset X$, we have $\phi = (s(L_{L^{2}}) - \mathcal{M})^{-1} D\mathcal{J}\phi \in X$. Hence, $s(L_{C}) \ge s(L_{L^{2}})> \max\limits_{x \in \overline{\Omega}} s(M(x))$. Following the same discussion as in Case 1  yields $s(L_{L^{2}}) \ge s(L_{C})$. Therefore, $s(L_{C}) = s(L_{L^{2}})$.
\end{proof}

By the classical variational formula (see \cite{MR361998,MR1544417}), we have the following
Rayleigh--Ritz characterization for $	s(L)$.

\begin{theorem}\label{t2.2}
	Assume that $L$ is a self-adjoint operator. Then
	\[
	s(L)= \lambda_{v}:=\sup_{\substack{\phi \in  L^{2}\\ \left\| \phi \right\| _{L^{2}}=1}} \int_{\Omega}\left[  \phi^{T}(x)D \left[ \mathcal{J\phi}\right](x) +  \phi^{T}(x) M(x) \phi(x) \right] \, \mathrm{d}x.
	\]
\end{theorem}

\begin{lemma}\label{l2.2}
	If there exists a  nonempty open set $\Omega_{0}\subset \Omega$ such that $s(M(x))=  \max\limits_{x \in \overline{\Omega}}s(M(x) )$ for all $x\in \Omega_{0}$, then $s(L)> \max\limits_{x \in \overline{\Omega}}s(M(x) )$.
\end{lemma}

\begin{proof}
	According to the Perron--Frobenius theorem, for each $x\in \overline{\Omega}$, $M(x)$ admits the principal eigenvalue. By the Kuratowski--Ryll--Nardzewski measurable selection theorem, we can choose $w\in L^{2}_{+}$ with $\max_{1\le i\le m} w_{i}(x)=1$ for all $x\in\overline\Omega$ such that
	$M(x)w(x)=s(M(x))w(x)$, $x\in\overline{\Omega}$.

	There exist $c_{0}>0$, $\delta_{0}>0$, $x_{0}\in\Omega_{0}$, and $\rho>0$ such that
	$J_{i}(x)>c_{0}$ for all $\|x\|_{\mathbb{R}^{n}}\le\delta_{0}$, $1\le i\le m$,
	and the ball
	$
	B(\rho):= \left\lbrace x \in \Omega_{0}: \left\| x-x_{0}\right\| \le\rho\right\rbrace  \subset \Omega_{0}
	$
	satisfying $\left\| x-y\right\| \le \delta_{0}$ for all $x, y \in B(\rho) $.
	Set
	$$
	\mathbb{I}: = \left\lbrace i : \int_{B(\frac{\rho}{2})} w_{i}(y)\, \mathrm{d}y>0 \right\rbrace ,
	$$
	and $\mathbb{J}:= \mathcal{S} \setminus\mathbb{I}$.
	We claim that $\mathbb{I} \neq \emptyset $.
	If not, then $\int_{B(\frac{\rho}{2})} w_{i}(y)\, \mathrm{d}y=0 $ for all $i$.
	It follows that $w_{i}(x)= 0$ almost everywhere on $x\in B(\frac{\rho}{2})$ for all $i$.
	Then the set
	$
	\Sigma:= \left\lbrace x\in B(\frac{\rho}{2}): w_{i}(x)>0 \text{ for some } i\right\rbrace
	$
	is a null set.
	However, this is impossible because, by the choice of $w$, the set $\Sigma$ should actually be the entire ball $B(\frac{\rho}{2})$. Let
	$$
	\sigma:=  \min\limits_{i\in \mathbb{I}}\int_{B(\frac{\rho}{2})} w_{i}(y)\, \mathrm{d}y>0
	$$
	and a cut-off function  $\eta(x)=1$ for $x\in B(\frac{\rho}{2})$ and $\eta(x)=0$ for $x\in \Omega\setminus  B(\frac{\rho}{2})$.
	For $x\in  B(\rho)$, we have
	\[
	d_{i}\int_{\Omega}J_{i}(x-y)\eta(y)w_{i}(y)\, \mathrm{d}y \ge d_{i}\int_{B(\frac{\rho}{2})}J_{i}(x-y)\eta(y)w_{i}(y)\, \mathrm{d}y\ge d_{i}c_{0}\sigma, \quad i\in \mathbb{I}.
	\]
	Then,
	$D\mathcal{J}\eta w\ge \underline d c_{0}\sigma\eta w$ in $L^{2}$ where $\underline d :=  \min\limits_{i\in \mathcal{S}}d_{i} $.
	Therefore, we have
	\[
	L_{L^{2}} \eta w \ge \left(  \underline d c_{0}\sigma+  \max\limits_{x \in \overline{\Omega}}s(M(x) )\right)   \eta w,
	\]
	which implies $s(L)>  \max\limits_{x \in \overline{\Omega}}s(M(x) )$.
\end{proof}

We note that Lemma \ref{l2.2} was proved in \cite[Lemma 2.2]{MR4601060} under the condition that $M(x)$ is irreducible for all $x$, where the core of the proof relies on the continuity of $w$.
The continuity of a continuous selection of principal eigenvectors depends crucially on the irreducibility of the matrix. If the matrix is reducible on an open set, such a selection may fail to be continuous. For example, let $A(x)=\operatorname{diag}(1+x,1-x)$ for $x\in[-\delta,\delta]$ with an arbitrary continuous extension outside this interval.  For $0<x<\delta$, the principal eigenspace is spanned by $e_1=(1,0)^T$, whereas for $-\delta<x<0$ it is spanned by $e_2=(0,1)^T$; at $x=0$, it is the whole plane.  Consequently, any choice of a principal eigenvector $w(x)$ has one-sided limits at $0$ lying in distinct directions, so $w$ cannot be continuous at $0$.

In contrast, we work in the $L^2$ setting, which does not require $w$ to be continuous; consequently, we do not need the irreducibility of $M(x)$. Our result thus extends the result by removing this condition.

\begin{theorem}\label{t2.1}
	Assume that {\rm(\textbf{M})} holds. Then
	$s(L)=\bar \lambda_{p} =\underline \lambda_{p}$.
\end{theorem}

\begin{proof}
	By the Touching Lemma (see \cite[Lemma 1.8]{MR4601060}), we can show that $\bar \lambda_{p} \ge \underline \lambda_{p}$. If $L_{C}$ admits the  principal eigenvalue, assumption (\textbf{M}) guarantees that the corresponding eigenfunction is strongly positive. Consequently, by the definitions of $\bar \lambda_{p}$ and  $\underline \lambda_{p}$, we have $\bar \lambda_{p} \le s(L) \le \underline \lambda_{p}$. Hence $s(L)=\bar \lambda_{p} =\underline \lambda_{p}$.

	Following the approximation method introduced in \cite[Lemma 4.2]{MR4929554} or \cite[Theorem A]{MR5005572}, for any $\varepsilon>0$, we can construct an operator $L^{\varepsilon}$ such that $L^{\varepsilon}$ admits the  principal eigenvalue. Indeed, set $\lambda =  \max\limits_{x \in \overline{\Omega}} s(M(x))$ and define
	$$
	\Omega_{\varepsilon} = \{x \in \bar{\Omega} : s(M(x)) \geq \lambda - \varepsilon\}.
	$$
	Then $\Omega_{\varepsilon}$ is a closed subset of $\bar{\Omega}$, and $s(M(x)) = \lambda - \varepsilon$ on $\partial\Omega_{\varepsilon} \cap \overline\Omega$. Next, define
	$$
	m_{ij}^{\varepsilon}(x) = m_{ij}(x), \ x \in \bar{\Omega}  \text{ for } i \neq j,
	$$
	and
	\[
	m_{ii}^{\varepsilon}(x) =
	\begin{cases}
		m_{ii}(x) - 2\varepsilon + \lambda - s(M(x)), & x \in \Omega_{\varepsilon}, \\
		m_{ii}(x) - \varepsilon, & x \in \bar{\Omega} \setminus \Omega_{\varepsilon}.
	\end{cases}
	\]
	Consequently, $m_{ii}^{\varepsilon}(x) \le m_{ii}(x) - \varepsilon$ for all $x \in \bar{\Omega}$. Set $\underline{M}^{\varepsilon}(x) = \bigl(m_{ij}^{\varepsilon}(x)\bigr)_{n \times n}$, i.e.,
	\[
	\underline{M}^{\varepsilon}(x) =
	\begin{cases}
		M(x) + [\lambda - 2\varepsilon - s(M(x))]I, & x \in \Omega_{\varepsilon}, \\
		M(x) - \varepsilon I, & x \in \bar{\Omega} \setminus \Omega_{\varepsilon}.
	\end{cases}
	\]
	Then $\underline{M}^{\varepsilon}(x)$ is continuous and cooperative, and $\underline{M}^{\varepsilon}$ is weakly irreducible. Moreover,
	\[
	\begin{cases}
		s\bigl(\underline{M}^{\varepsilon}(x)\bigr) = s(M(x)) + \lambda - 2\varepsilon - s(M(x)) = \lambda - 2\varepsilon, & x \in \Omega_{\varepsilon}, \\
		s\bigl(\underline{M}^{\varepsilon}(x)\bigr) = s(M(x)) - \varepsilon < \lambda - 2\varepsilon, & x \in \bar{\Omega} \setminus \Omega_{\varepsilon}.
	\end{cases}
	\]
	Finally, let $L^{\varepsilon}$ denote the operator defined analogously to $L$ but with $M$ replaced by $\underline{M}^{\varepsilon}$. Then by Lemma \ref{l2.2}, $L^{\varepsilon}$ admits a principal eigenvalue.
	Notice that $\underline{M}^{\varepsilon}(x) \le M(x) \le \underline{M}^{\varepsilon}(x) +2\varepsilon$.
	By the monotonicity on $M$, we have
	\begin{equation}\label{eq2.1}
		\bar \lambda_{p}^{\varepsilon} \le \bar \lambda_{p} \le \bar \lambda_{p}^{\varepsilon}+ 2\varepsilon, \quad
		\underline \lambda_{p}^{\varepsilon}\le \underline \lambda_{p}\le \underline \lambda_{p}^{\varepsilon}+2\varepsilon, \quad
		s(L^{\varepsilon}) \le s(L)\le s(L^{\varepsilon})+2\varepsilon,
	\end{equation}
	and $s(L^{\varepsilon}) $ is strictly decreasing in $\varepsilon$, where $\bar{\lambda}_{p}^{\varepsilon}$ and $\underline{\lambda}_{p}^{\varepsilon}$ denote the analogues of $\bar{\lambda}_{p}$ and $\underline{\lambda}_{p}$ where $L$ is replaced by $L^{\varepsilon}$, respectively.
	Letting $\varepsilon\to 0$ in \eqref{eq2.1}, since $s(L^{\varepsilon})=\bar \lambda_{p}^{\varepsilon} =\underline \lambda_{p}^{\varepsilon} $, we obtain that $s(L)=\bar \lambda_{p} =\underline \lambda_{p} $.
\end{proof}

\begin{remark}
The equality \(\bar{\lambda}_p = \underline{\lambda}_p\) was previously established for cases where \(M\) is strongly irreducible \cite{MR4929554,MR4601060}; moreover, in the self-adjoint case with \(M\) strongly irreducible, it is known that \(\bar{\lambda}_p = \underline{\lambda}_p = \lambda_v\) \cite{MR4929554}. In contrast, our result holds under the significantly weaker assumption \((\mathbf{M})\). Furthermore, whereas the existing results only establish the equivalence of the two generalized eigenvalues, we prove that these generalized eigenvalues coincide with the spectral bound of the operator. This identification reveals the intrinsic nature of the generalized eigenvalue and lays the groundwork for the subsequent study of the weighted eigenvalue problem.
\end{remark}

\section{The weighted eigenvalue problem}

In this section, we establish variational characterizations of \(\mu_0\) as the unique value satisfying \(s(\mathcal{L}_{\mu_0})=0\), where Theorems \ref{t2.2} and \ref{t2.1} play a key role.  Here, we provide a new analytical approach that allows us to overcome the absence of a principal eigenvalue and still obtain the characterization of \(\mu_0\), thereby providing a new perspective on spectral problems for nonlocal dispersal operators.

To this end, we first prove the monotonicity and continuity of \(s(\mathcal{L}_{\mu})\) with respect to \(\mu\), and then derive necessary and sufficient conditions for the existence of a unique \(\mu_0\) such that \(s(\mathcal{L}_{\mu_0})=0\). Building on the Collatz--Wielandt and Rayleigh--Ritz characterizations of the spectral bound developed in Section 2, together with the monotonicity of \(s(\mathcal{L}_{\mu})\), we then establish the corresponding characterizations for this weighted eigenvalue problem.

Note that, as before, $\mathcal{L}_{\mu}$ can be regarded as an operator on $X$ and on $L^{2}$. By Lemma \ref{l2.1}, these operators have the same spectral bound, and hence we do not distinguish between them.
We make some  assumptions on $J$, $A$ and $F$:
\begin{itemize}
	\item [(\textbf{S})] $J_{i}(x)=J_{i}(-x)$ for all $i\in \mathcal{S}$;  $A(x)$ and $F(x)$ are  symmetric matrices;
	
	\item[(\textbf{R})] $A+F$ is weakly irreducible.

\end{itemize}
Note that $A+F$ is weakly irreducible implies $A+\frac{1}{\mu}F$ is weakly irreducible for all $\mu>0$.
In this section, we continue to assume that (\textbf{J})   and (\textbf{D}) hold.

\begin{proposition}\label{pr3.1}
	Assume that {\rm (\textbf{R})}	holds. Then $s(\mathcal{L}_{\mu})$ is continuous with respect to $\mu>0$.
\end{proposition}

\begin{proof}
	According to Theorem \ref{t2.1}, for any $\varepsilon>0$, there exists $\phi \in \mathrm{Int}(X_{+}) $ such that
	$
	s(\mathcal{L}_{\mu})\phi(x)+\frac{\varepsilon}{2}\phi(x)\ge \left[ \mathcal{L}_{\mu} \phi\right] (x)$ for $ x\in \overline{\Omega}.
	$
	It follows that
	\[
	\begin{aligned}
		\left[ \mathcal{L}_{\mu^{'}} \phi\right] (x)\le & s(\mathcal{L}_{\mu})\phi(x)+\dfrac{\varepsilon}{2}\phi(x) + \left( \frac{1}{\mu'} -\frac{1}{\mu}\right) F(x)\phi(x)\\
		\le &
		s(\mathcal{L}_{\mu})\phi(x)+\dfrac{\varepsilon}{2}\phi(x) + \dfrac{\mu-\mu^{'}}{\mu \mu^{'}}C\phi(x)\\
		\le & s(\mathcal{L}_{\mu})\phi(x)+\varepsilon\phi(x),
	\end{aligned}
	\]
	provided $ 0\le \mu -\mu^{'}\le \frac{\mu^{2}\varepsilon}{2C+\mu \varepsilon}$, where $C=\frac{m  \max\limits_{x \in \overline{\Omega}, i,j\in \mathcal{S}}f_{ij}(x)  \max\limits_{x \in \overline{\Omega}, i\in \mathcal{S}}\phi_i(x)}{ \min\limits_{x\in\overline{\Omega}, i\in \mathcal{S}}\phi_i(x)}>0$.
	Then, $s(\mathcal{L}_{\mu^{'}})\le s(\mathcal{L}_{\mu}) + \varepsilon$ provided $ 0\le \mu -\mu^{'}\le \frac{\mu^{2}\varepsilon}{2C+\mu \varepsilon}$.
	Similarly, there is $\psi\in\mathrm{Int}(X_{+}) $ such that
	$
	s(\mathcal{L}_{\mu})\psi(x)-\dfrac{\varepsilon}{2}\psi(x)\le \left[ \mathcal{L}_{\mu} \psi\right] (x)$  for $ x\in \overline{\Omega}.
	$
	Then
	\[
	\begin{aligned}
		\left[ \mathcal{L}_{\mu^{'}} \psi\right] (x)\ge & s(\mathcal{L}_{\mu})\psi(x)-\dfrac{\varepsilon}{2}\psi(x) + \left( \frac{1}{\mu'} -\frac{1}{\mu}\right) F(x)\psi(x)\\
		\ge &
		s(\mathcal{L}_{\mu})\psi(x)-\dfrac{\varepsilon}{2}\psi(x) + \dfrac{\mu-\mu^{'}}{\mu \mu^{'}}C^{'}\psi(x)\\
		\ge & s(\mathcal{L}_{\mu})\psi(x)-\varepsilon\psi(x),
	\end{aligned}
	\]
	provided $ 0\ge \mu -\mu^{'}\ge -\frac{\mu^{2}\varepsilon}{2C^{'}}$, where
	$C^{'}=\frac{m  \max\limits_{x \in \overline{\Omega}, i,j\in \mathcal{S}}f_{ij}(x)  \max\limits_{x \in \overline{\Omega}, i\in \mathcal{S}}\psi_i(x)}{ \min\limits_{x\in\overline{\Omega}, i\in \mathcal{S}}\psi_i(x)}>0$.
	It follows that
	$s(\mathcal{L}_{\mu^{'}})\ge s(\mathcal{L}_{\mu}) -\varepsilon$ provided $ 0\ge \mu -\mu^{'}\ge- \frac{\mu^{2}\varepsilon}{2C^{'}}$.
	It is straightforward to verify that $s(\mathcal{L}_{\mu})$ is nonincreasing in $\mu >0$.
	Therefore, there exists $\delta:=\min\left\lbrace \frac{\mu^{2}\varepsilon}{2C^{'}},  \frac{\mu^{2}\varepsilon}{2C+\mu \varepsilon}  \right\rbrace $ such that $\left| s(\mathcal{L}_{\mu^{'}})-s(\mathcal{L}_{\mu}) \right| \le \varepsilon$ if $\left| \mu - \mu^{'}\right| \le \delta$.
\end{proof}

\begin{proposition}\label{pr3.4}
	$\lim\limits_{\mu\to 0^{+}} s(\mathcal{L}_{\mu})= +\infty$.
\end{proposition}

\begin{proof}
	Notice that
	$s(\mathcal{L}_{\mu})\ge  \max\limits_{x \in \overline{\Omega}}s(A(x)+\frac{1}{\mu}F(x))$. Since $A(x)$ is a cooperative  matrix, there exists  $c>0$ such that $A(x)+cI $ is a nonnegative   matrix. It follows that  $A(x)+cI+\frac{1}{\mu}F(x)\ge \frac{1}{\mu}F(x)$,
	and
	$s(A(x)+\frac{1}{\mu}F(x))\ge s(\frac{1}{\mu}F(x))-c$.
	Since $\max_{x\in \overline{\Omega}}s(F(x))>0$, one can prove that $\lim\limits_{\mu\to 0^{+}} s(\frac{1}{\mu}F(x)) = +\infty$ for some $x\in \overline{\Omega}$.
	Therefore, $\lim\limits_{\mu\to 0^{+}} s(\mathcal{L}_{\mu})= +\infty$.
\end{proof}

Next, we establish a necessary and sufficient condition for the existence of a unique zero solution to $s(\mathcal{L}_{\mu}) $. Define the operators $  \mathcal{L}_{\infty}: X \to X$ and $ \mathcal{F} : X \to X$ by
\[
\left[ \mathcal{L}_{\infty} \phi\right] (x) := D \left[ \mathcal{J}\phi\right](x) + A(x)\phi(x), \quad \left[ \mathcal{F}\phi\right] (x):= F(x)\phi(x), \quad  \phi\in X.
\]
According to \cite[Theorem 3.5]{MR2505085}, we have the following conclusion.

\begin{proposition}\label{pr3.3}
	Assume that $s( \mathcal{L}_{\infty}) <0$. Then $s(\mathcal{L}_{\mu})$ has the same sign as $\frac{1}{\mu}r(-\mathcal{F} \mathcal{L}_{\infty}^{-1}) -1$.
\end{proposition}

\begin{theorem}\label{t3.2}
	There exists a unique $\mu_{0}>0$ such that $s(\mathcal{L}_{\mu_{0}})=0 $ if and only if    $s(\mathcal{L}_{\infty})<0$.
\end{theorem}

\begin{proof}
	We first prove the necessity.
	Since $s(\mathcal{L}_{\mu})$ is nonincreasing in $\mu$ and the uniqueness of $\mu_{0}>0$, we have $s(\mathcal{L}_{\mu})<0 $ whenever $\mu > \mu_{0}$. It follows that $s(\mathcal{L}_{\infty})<0$.

	Next, we prove the sufficiency. By the perturbation theory for linear operators (see, e.g., \cite[Section IV.3]{MR203473}), there exists \(\mu_* > 0\) such that \(s(\mathcal{L}_\mu) < 0\) for all \(\mu \ge \mu_*\). Then the existence of \(\mu_0\) follows directly from Propositions \ref{pr3.1} and  \ref{pr3.4}. Uniqueness is guaranteed by Proposition \ref{pr3.3}.
\end{proof}

We now prove the main theorem of this section, establishing variational characterizations of \(\mu_0\). The proof relies on the characterizations of the spectral bound developed in Section 2. Set
\[
\mathcal{X}:= \left\lbrace \phi \in \mathrm{Int}(X_{+}) :  \left[  \mathcal{L}_{\infty}\phi\right](x)  <0, \forall x\in \overline{\Omega} \right\rbrace,
\]
and for any $\phi \in \mathrm{Int}(X_{+})$, let
\[
\mathscr{D}_{\phi}:= \left\lbrace  (x,i)\in \overline{\Omega} \times\mathcal{S} :  -d_i \int_{\Omega} J_i(x- y) \phi_i(y) \,\mathrm{d}y - \sum_{j=1}^{m}   a_{ij}(x)\phi_{j}(x) >0\right\rbrace.
\]
Under the assumption $s(\mathcal{L}_{\infty})<0$, it is easy to verify that $\mathcal{X}$ and $ \mathscr{D}_{\phi}$ are nonempty. Indeed, for $\mu> \mu_{0}$, we have $s( \mathcal{L}_{\mu})<0$. By Theorem \ref{t2.1}, there exists  $\phi\in \mathrm{Int}(X_{+})$ satisfying $ \left[  \mathcal{L}_{\mu}\phi\right](x)<0, \forall  x\in \overline{\Omega}$. It follows that $\left[  \mathcal{L}_{\infty}\phi\right](x)\le\left[  \mathcal{L}_{\mu}\phi\right](x)<0$. Hence, $\mathcal{X}\neq \emptyset $. To show $\mathscr{D}_{\phi}\neq \emptyset$,  suppose for contradiction that
$\mathscr{D}_{\phi} =\emptyset$. Then, there exists some $\phi\in \mathrm{Int}(X_{+}) $ satisfying $\left[  \mathcal{L}_{\infty}\phi\right](x)\ge 0,  \forall  x \in \overline{\Omega}$. This implies $s(\mathcal{L}_{\mu})\ge 0$ for all $\mu>0$, a contradiction. 

\begin{theorem}\label{2.1}
	Assume that {\rm(\textbf{R})} holds and there exists a unique  $\mu_{0}>0$ such that $s(\mathcal{L}_{\mu_{0}})=0 $. Then
	\[
	\begin{aligned}
		\mu_{0}= &\inf_{\phi \in \mathcal{X}} \sup_{x \in \Omega, i \in \mathcal{S}} \dfrac{\sum_{j=1}^{m}f_{ij}(x)\phi_{j}(x)}{-d_i \int_{\Omega} J_i(x- y) \phi_i(y) \,\mathrm{d}y - \sum_{j=1}^{m}  a_{ij}(x)\phi_{j}(x)}
		\\
		=& \sup_{\phi \in \mathrm{Int}(X_{+}) } \inf_{(x, i) \in \mathscr{D}_{\phi}} \dfrac{\sum_{j=1}^{m}f_{ij}(x)\phi_{j}(x)}{-d_i \int_{\Omega} J_i(x- y) \phi_i(y) \,\mathrm{d}y - \sum_{j=1}^{m}   a_{ij}(x)\phi_{j}(x)}.
	\end{aligned}
	\]
\end{theorem}

\begin{proof}
	Let
$$
	\mu_{1}= \inf\left\lbrace \mu>0: \exists \phi \in \mathrm{Int}(X_{+}) \text{ such that }  \left[ \mathcal{L}_{\mu}\phi\right](x)\le 0,  \forall x \in \overline{\Omega}  \right\rbrace,
	$$
and
$\mu_{2} = \sup\left\lbrace \mu>0:  \exists \phi \in \mathrm{Int}(X_{+}) \text{ such that } \left[ \mathcal{L}_{\mu}\phi\right](x)\ge 0,  \forall x \in \overline{\Omega} \right\rbrace$.
We first show that $\mu_{1}=\mu_{2}=\mu_{0}$, and then prove $\mu_{1}$ and $\mu_{2}$ admit the  inf–sup and sup–inf expression as stated in the theorem, respectively.
Recall that
\[
	\begin{aligned}
		s( \mathcal{L}_{\mu}) = & \inf \left\lbrace \lambda\in \mathbb{R}: \exists \phi \in \mathrm{Int}(X_{+}) \ s.t. \ \left[ \mathcal{L}_{\mu}\phi\right](x)\le \lambda \phi(x),  \forall x \in \overline{\Omega}  \right\rbrace \\
		= & \sup \left\lbrace \lambda\in \mathbb{R}: \exists \phi \in \mathrm{Int}(X_{+}) \ s.t. \ \left[ \mathcal{L}_{\mu}\phi\right](x)\ge \lambda \phi(x),  \forall x \in \overline{\Omega}  \right\rbrace.
	\end{aligned}
	\]
Since $s(\mathcal{L}_{\mu})$ is nonincreasing for $\mu >0$ and there exists a unique $\mu_{0}>0$ such that $s(\mathcal{L}_{\mu_{0}})=0$, it follows that $s(\mathcal{L}_{\mu})<0$ for all $\mu > \mu_{0}$ and $s(\mathcal{L}_{\mu})>0$ for all $\mu < \mu_{0}$. Consequently, for any $\mu > \mu_{0}$, we can find $\psi\in \operatorname{Int}(X_{+})$ satisfying $\mathcal{L}_{\mu} \psi\le 0$, which implies $\mu_{0}\ge \mu_{1}$. If $\mu_{0}> \mu_{1}$, by the definition of $\mu_{1}$, there exist some $\mu'\in (\mu_{1}, \mu_{0})$ and $\varphi \in \operatorname{Int}(X_{+})$ such that $\mathcal{L}_{\mu'} \varphi\le 0$. This would give $s(\mathcal{L}_{\mu'}) \le 0$,  a contradiction. Therefore, $\mu_{0}=\mu_{1}$.

Similarly, for any $\mu < \mu_{0}$, there exists $\psi\in \operatorname{Int}(X_{+})$ with $\mathcal{L}_{\mu} \psi\ge 0$, meaning $\mu_{0}\le \mu_{2}$. If $\mu_{0}< \mu_{2}$, then one can find $\mu'\in ( \mu_{0}, \mu_{2})$ and $\varphi \in \operatorname{Int}(X_{+})$ such that $\mathcal{L}_{\mu'} \varphi\ge 0$, which would imply $s(\mathcal{L}_{\mu'}) \ge 0$, again a contradiction. Hence, $\mu_{0}=\mu_{2}$.

Denote by $\mu_{1}^{'}$ and $\mu_{2}^{'}$ the inf-sup and sup-inf expressions from the theorem, respectively.

For the inf–sup quantity: If $\mu > \mu_{1}^{'}$, we can verify that  $\mu \ge \mu_{1}$, so $\mu_{1}^{'} \ge \mu_{1} $.  Conversely, for any  $\mu > \mu_{1}=\mu_{0}$, there exists $\phi \in \operatorname{Int}(X_{+})$ such that $\left[ \mathcal{L}_{\mu}\phi\right](x) <0 $ on $ \overline{\Omega}$, which gives  $\phi\in \mathcal{X}$ and $\mu \ge \mu_{1}^{'}$. Hence $\mu_{1}\ge \mu_{1}^{'} $. Therefore $\mu_{1}= \mu_{1}^{'}$.

For the sup–inf expression: 
If  $0< \mu< \mu_{2}^{'}$, we have $\left[ \mathcal{L}_{\mu}\phi \right](x) \ge 0 $ for some $\phi \in \operatorname{Int}(X_{+})$. Then $\mu\le \mu_{2}$, so $ \mu_{2}^{'} \le \mu_{2}$. Conversely, for any $0<\mu < \mu_{2}$, we can verify that  $\mu \le \mu_{2}^{'}$. Hence $\mu_{2}\le \mu_{2}^{'}$. Therefore $\mu_{2}=\mu_{2}^{'}$.

\end{proof}

\begin{theorem}\label{2.3}
	Assume that  {\rm{(\textbf{S})}} holds and there exists a unique  $\mu_{0}>0$ such that $s(\mathcal{L}_{\mu_{0}})=0 $. Then
	\[
	\mu_{0}= \sup_{\substack{\phi \in  L^{2} \\  \left\| \phi \right\| _{L^{2}}=1}} \dfrac{\int_{\Omega}\phi^{T}(x)F(x)\phi(x)\,\mathrm{d}x}{-\int_{\Omega}\left[  \phi^{T}(x)D \left[ \mathcal{J\phi}\right](x) +  \phi^{T}(x) A(x) \phi(x)\right]\,\mathrm{d}x }.
	\]
\end{theorem}

\begin{proof}
	
	Denote the right-hand side of the above equality by $\mu'$.
	For any $0<\mu < \mu'$, there exists some $\phi \in  L^{2}$ with $ \left\| \phi \right\| _{L^{2}}=1$ such that
	\[
	\mu<\dfrac{\int_{\Omega}\phi^{T}(x)F(x)\phi(x)\,\mathrm{d}x}{-\int_{\Omega}\left[  \phi^{T}(x)D \left[ \mathcal{J\phi}\right](x) +  \phi^{T}(x) A(x) \phi(x)\right]\,\mathrm{d}x}.
	\]
	This implies
	\[
	\int_{\Omega}\left[  \phi^{T}(x)D \left[ \mathcal{J\phi}\right](x) +  \phi^{T}(x) A(x) \phi(x) + \frac{1}{\mu}\phi^{T}(x) F(x) \phi(x)\right] \, \mathrm{d}x>0,
	\]
	and consequently $s(\mathcal{L}_{\mu})>0$. By the monotonicity of $\mu$, we obtain $\mu < \mu_0$. Since $\mu$ was arbitrary, it follows that $\mu' \le \mu_0$.
	
	Now take any $\mu > \mu'$. Then
	\[
	\mu > \dfrac{\int_{\Omega}\phi^{T}(x)F(x)\phi(x)\,\mathrm{d}x}{-\int_{\Omega}\left[  \phi^{T}(x)D \left[ \mathcal{J\phi}\right](x) +  \phi^{T}(x) A(x) \phi(x)\right]\,\mathrm{d}x}, \quad \forall \phi \in  L^{2} \text{ with }  \left\| \phi \right\| _{L^{2}}=1.
	\]
	Hence, for every such $\phi $, we have
	\[
	\int_{\Omega}\left[  \phi^{T}(x)D \left[ \mathcal{J\phi}\right](x) +  \phi^{T}(x) A(x) \phi(x) + \frac{1}{\mu}\phi^{T}(x) F(x) \phi(x)\right] \, \mathrm{d}x<0,
	\]
	which gives $s(\mathcal{L}_{\mu})\le0 $. Thus $\mu\ge \mu_{0}$, and by the arbitrariness of $\mu$, we conclude  $\mu'\ge \mu_{0} $. Combining both inequalities yields $\mu_0 = \mu'$, which completes the proof.
\end{proof}

\begin{remark}\label{re1}
We point out that Theorem \ref{2.3} also holds in the partially degenerate case, where assumption {\rm(\textbf{D})} is weakened to {\rm(\textbf{D$'$})}:
\begin{itemize}
	\item[\rm(\textbf{D$'$})] There exists some \(1 \le k<m\) such that \(d_{i}>0\) for all \(1\le i\le k\), and \(d_{i}=0\) for all \(k< i\le m\).
\end{itemize}
Moreover, we can extend Theorem \ref{2.1} to this partially degenerate case by strengthening the irreducibility assumptions on \(A\) and \(F\):
\begin{itemize}
	\item[\rm(\textbf{RR})] \(A+F\) is strongly irreducible, meaning that \(A(x)+F(x)\) is irreducible for every \(x\in \overline{\Omega}\).
\end{itemize}
Indeed, under the assumption that \(M\) is strongly irreducible, an analogue of Lemma \ref{l2.2} has been established in \cite[Theorem A]{MR4803717}. Hence, Theorem \ref{t2.1} remains valid in this case, and consequently Theorem \ref{2.1} follows.
Furthermore, in the partially degenerate setting, Theorem \ref{t3.2} remains valid even though Proposition \ref{pr3.1} no longer holds. Indeed, according to Propositions \ref{pr3.4} and \ref{pr3.3}, we obtain \(r(-\mathcal{F}\mathcal{L}_{\infty}^{-1})>0\). Consequently, the existence and uniqueness of \(\mu_0\) are guaranteed by Proposition \ref{pr3.3}.
\end{remark}

\section{The basic reproduction ratio}

\subsection{General nonlocal dispersal systems}

In this subsection, we apply the results established in Section 3 to characterize the basic reproduction ratio. To this end, we consider the following general nonlocal dispersal systems:
\begin{equation}\label{eq5.1}
	\begin{cases}
		\dfrac{	\partial u}{\partial t}= D[\mathcal{J}u(t,\cdot)](x) + A(x)u(t,x) +  F(x)u(t,x), & t>0,  x \in \Omega,\\
		u(0,x)=u_{0}(x), & x\in \Omega,
	\end{cases}
\end{equation}
and the corresponding system in the absence of new infections:
\begin{equation}\label{eq5.2}
	\begin{cases}
		\dfrac{	\partial u}{\partial t}= D[\mathcal{J}u(t,\cdot)](x) + A(x)u(t,x), & t>0,  x \in \Omega,\\
		u(0,x)=u_{0}(x), & x\in \Omega.
	\end{cases}
\end{equation}
By the general theory of semigroups, the system \eqref{eq5.2} generates  a $C_0$-semigroup $\Phi(t)$ on $X$ such that $\left[ \Phi(t)u_{0}\right](x)= u(t,x)$, where $u(t,x)$ denotes the solution to system \eqref{eq5.2} with the initial condition $u(0, x)=u_0(x)$. Define the operators $Q, \hat Q: X \to X$ by
\[
[Qu](x):= \int_{0}^{+\infty} F(x)[\Phi(t)u](x)\,\mathrm{d}t, \qquad
[\hat Q u](x):= \int_{0}^{+\infty} [\Phi(t)\mathcal{F}u](x)\,\mathrm{d}t, \quad u\in X.
\]

To apply the theory of the basic reproduction ratio, as developed in \cite{MR3032845,MR3612966,MR3992071,MR2505085}, we impose the following assumption:
\begin{itemize}
	\item [(\textbf{L})] $s(\mathcal{L}_\infty)<0$.
\end{itemize}
Clearly, if (\textbf{L}) holds, then $Q= -\mathcal{F}\mathcal{L}_{\infty}^{-1}$ and $\hat Q  =  -\mathcal{L}_{\infty}^{-1}\mathcal{F}$; see \cite[Theorem 3.12]{MR2505085}.
It is easy to verify that $r(Q)=r(\hat Q)$ by  Gelfand's formula (see, e.g., \cite[Theorem VI.6]{MR751959}).
The spectral radius of $Q$ or $\hat Q$ is then defined as the basic reproduction ratio for system \eqref{eq5.1}, i.e.,
$$
\mathcal{R}_{0}= r(Q) \text{ or } \mathcal{R}_{0} =r (\hat Q).
$$

According to \cite[Theorem 3.5]{MR2505085} (or Proposition \ref{pr3.3}) together with Theorem \ref{t3.2} and Remark \ref{re1}, we obtain the following result.

\begin{proposition}\label{pr5.1}
	Assume that {\rm(\textbf{J})} and {\rm(\textbf{L})} hold and that  either  {\rm(\textbf{D})} or  {\rm(\textbf{D}$'$)} is satisfied. Then $\mathcal{R}_{0}-1$ has the same sign as $s(\mathcal{L}_{1})$, and $ \mathcal{R}_{0}$ is the unique positive solution of $s(\mathcal{L}_{\mu})=0$.
\end{proposition}

Applying Theorems \ref{2.1} and \ref{2.3}, and Remark \ref{re1}, we obtain the main results of this section.

\begin{theorem}\label{th5.1}
	Assume that {\rm(\textbf{J})} and {\rm(\textbf{L})} hold. The following statements are valid:
	
	\begin{itemize}
		\item [\rm(i)] If {\rm(\textbf{S})} and either {\rm(\textbf{D})} or {\rm(\textbf{D}$'$)} hold, then
		\[
		\mathcal{R}_{0}= \sup_{\substack{\phi \in  L^{2}\\ \left\| \phi \right\| _{L^{2}}=1}} \dfrac{\int_{\Omega}\phi^{T}(x)F(x)\phi(x)\,\mathrm{d}x}{-\int_{\Omega}\left[  \phi^{T}(x)D \left[ \mathcal{J\phi}\right](x) +  \phi^{T}(x) A(x) \phi(x)\right]\,\mathrm{d}x }.
		\]
		
		\item[\rm(ii)]  If either {\rm(\textbf{D})} with {\rm(\textbf{R})} holds, or {\rm(\textbf{D}$'$)} with {\rm(\textbf{RR})} holds, then
		\[
		\begin{aligned}
			\mathcal{R}_{0}= &\inf_{\phi \in \mathcal{X}  } \sup_{x \in \Omega, i \in \mathcal{S}} \dfrac{\sum_{j=1}^{m}f_{ij}(x)\phi_{j}(x)}{-d_i \int_{\Omega} J_i(x- y) \phi_i(y) \,\mathrm{d}y - \sum_{j=1}^{m}  a_{ij}(x)\phi_{j}(x)}
			\\
			=& \sup_{\phi \in \mathrm{Int}(X_{+}) } \inf_{(x,i)\in \mathscr{D}_{\phi}} \dfrac{\sum_{j=1}^{m}f_{ij}(x)\phi_{j}(x)}{-d_i \int_{\Omega} J_i(x- y) \phi_i(y) \,\mathrm{d}y - \sum_{j=1}^{m}   a_{ij}(x)\phi_{j}(x)}.
		\end{aligned}
		\]

	\end{itemize}

\end{theorem}

\subsection{SIS model with saturated incidence}

In this subsection, we apply the general characterizations obtained in Subsection 4.1 to a nonlocal dispersal SIS epidemic model with saturated incidence, and investigate how the model parameters affect the basic reproduction ratio, particularly when the saturation coefficient degenerates. Consider the following model:
\begin{equation}\label{eq6.1}
	\begin{cases}
	\displaystyle	\frac{\partial S}{\partial t} = d_S \int_{\Omega} J(x-y)[S(t,y) - S(t,x)] \, \mathrm{d}y - \frac{\beta(x)S I}{m(x) + S + I} + \gamma(x)I, & x \in \Omega, \, t > 0, \\[1em]
	\displaystyle	\frac{\partial I}{\partial t} = d_I \int_{\Omega} J(x-y)[I(t,y) - I(t,x)] \, \mathrm{d}y + \frac{\beta(x)S I}{m(x) + S + I} - \gamma(x)I, & x \in \Omega, \, t > 0, \\[1em]
	S(x,0) = S_0(x), \, I(x,0) = I_0(x), & x \in \Omega,
	\end{cases}
\end{equation}
where $\Omega \subset \mathbb{R}^n$ is a smooth bounded domain; $S(t,x)$ and $I(t,x)$ denote the densities of susceptible and infectious individuals at location $x \in \Omega$ and time $t > 0$, respectively; the positive constants $d_S$ and $d_I$ are dispersal coefficients for susceptible and infectious individuals, respectively; $\beta, \gamma\in C(\overline{\Omega},\mathbb{R}_{+})$  represent the transmission rate and recovery rate at $ x \in \Omega$, respectively; $m \in C(\overline{\Omega},\mathbb{R}_{+})$  is the saturation coefficient, which reflects the finite contact capacity of individuals due to spatial or social constraints (see \cite{MR1118756}).
The integral operator $\int_{\Omega}J(x-y)[u(t,y) - u(t,x)]\, \mathrm{d}y$ represents a diffusion process that is confined to the domain $\Omega$. This confinement implies that no individuals can enter or leave the domain, which corresponds to the imposition of homogeneous Neumann  boundary conditions; see, e.g., \cite{MR2722295,MR3637938,MR3945624}.
We assume that
\begin{itemize}
	
	\item[(\textbf{A1})]   $J$ satisfies assumptions (\textbf{J}) and (\textbf{S});
	
	\item[(\textbf{A2})] $\gamma$ and  $\beta$ are continuous nonnegative functions with $\gamma(x)\not\equiv 0$ and $\beta(x)\not\equiv 0$;
	
	\item[(\textbf{A3})] $m$ is a continuous positive  function;
	
	\item[(\textbf{A3}$'$)] $m$ is a continuous nonnegative function such that the set $\varSigma:=\left\lbrace x\in \overline{\Omega}: m(x)=0 \right\rbrace $ is  nonempty.
\end{itemize}

Throughout this section, we always assume that (\textbf{A1}) and (\textbf{A2}) hold.
Clearly,
$$
\frac{\partial }{\partial t}\int_{\Omega}\left( S(t,x)+I(t,x)\right) \mathrm{d}x =0.
$$
Therefore, we denote $K:=\frac{1}{\left| \Omega\right| }\int_{\Omega}\left( S(t,x) \! + \! I(t,x)\right)  \mathrm{d}x $, and the disease-free equilibrium of the system \eqref{eq6.1} is $(K,0)$.

For $\beta(x)>0 $, $\gamma(x)>0$ and $m(x)>0$, system \eqref{eq6.1} was studied by Feng et al. \cite{Feng_Li_Yang_2024}. They established the threshold dynamics of the system in terms of the basic reproduction ratio and further examined how the diffusion coefficients influence both the basic reproduction ratio and the existence of endemic equilibrium. We note that under assumptions (\textbf{A1}) and (\textbf{A2}), the same threshold dynamics in terms of the basic reproduction ratio can also be established by a similar argument. Therefore, we do not repeat this analysis here and instead focus on the properties of the basic reproduction ratio itself, such as its variational characterizations and asymptotic behavior.  For models with random (local) diffusion instead of nonlocal dispersal, we refer to \cite{MR4510469,MR4687008,MR4675083} and the references therein.

Consider the following nonlocal dispersal system:
\begin{equation}\label{eq6.2}
	\begin{cases}
	\displaystyle	\dfrac{\partial u}{\partial t}= d_{I}  \int_{\Omega} J(x-y)[u(t,y) - u(t,x)] \, \mathrm{d}y + \frac{K\beta(x)}{m(x) + K}u - \gamma(x)u, & x \in \Omega, \, t > 0, \\
		u(x,0) = u_0(x), & x \in \Omega.
	\end{cases}
\end{equation}
Let $Y:= C(\overline\Omega, \mathbb{R})$ equipped with the maximum norm, and $Y_{+}:= C(\overline\Omega, \mathbb{R}_{+})$.
By the general theory of semigroups, the system \eqref{eq6.2} generates  a $C_0$-semigroup $\Psi(t)$ on $Y$ such that $\left[ \Psi(t)u_{0}\right](x)= u(t,x)$, where $u(t,x)$ denotes the solution to system \eqref{eq6.2} with the initial condition $u(0, x)=u_0(x)$. Define an operator $P:Y\to Y$ by
\[
\left[ Pu\right] (x):= \int_{0}^{+\infty}\frac{K\beta(x)}{m(x) + K}\left[ \Psi(t)u\right](x) \, \mathrm{d}t, \quad u \in Y.
\]
The spectral radius of $P$ is then defined as the basic reproduction ratio for system \eqref{eq6.1}, i.e.,
$$
\mathcal{R}_{0}= r(P).
$$
We define the operators $\mathcal{\hat L}_{\mu} : Y \to Y$  and $\mathcal{\hat L}_{\infty} : Y\to Y$ as follows:
\[
\left[ \mathcal{\hat L}_{\mu}u\right] (x):= d_{I} \int_{\Omega} J(x-y)\left[ u(y)-u(x)\right] \, \mathrm{d}y  -\gamma(x) u(x) +\frac{K\beta(x)}{\mu \left( m(x) + K\right) } u(x), \quad u\in Y,
\]
and
$$
\left[ \mathcal{\hat L}_{\infty}u\right] (x):= d_{I} \int_{\Omega} J(x-y)\left[ u(y)-u(x)\right] \, \mathrm{d}y  -\gamma(x) u(x), \ u\in Y.
$$
Recall that
\[
s(\mathcal{\hat L}_{\infty})= \sup\limits_{\substack{u\in L^{2}(\Omega) \\  \left\| u \right\| _{L^{2}}=1} } \left\lbrace -\frac{d_{I}}{2} \int_{\Omega} \int_{\Omega} J(x - y)\left[ u(y) - u(x)\right] ^2 \, \mathrm{d}y\mathrm{d}x-\int_{\Omega}\gamma(x)u^{2}(x)\, \mathrm{d}x \right\rbrace.
\]
Denote
\[
\alpha:=\mathop{\inf}\limits_{\substack{ u\in L^{2}(\Omega),u\neq0 \\ \int_{\Omega}u(x)\,\mathrm{d}x=0}} \dfrac{\frac{1}{2}\int_{\Omega}\int_{\Omega}J(x-y)\left[u(y)-u(x)\right]^{2}\, \mathrm{d}y\mathrm{d}x }{\int_{\Omega}u^{2}(x)\, \mathrm{d}x}.
\]		
According to \cite[Proposition 3.4 and Lemma 3.5]{MR2722295}, we have
\begin{equation}\label{eq2.4}
	0<\alpha\le \mathop{\min}\limits_{x\in \overline{\Omega}}\int_{\Omega}J(x-y)\,\mathrm{d}y.
\end{equation}

\begin{proposition}\label{pr2.2}
	
	$s(\mathcal{\hat L}_{\infty})<0$.
	
\end{proposition}

\begin{proof}
	It is easy to see $s(\mathcal{\hat L}_{\infty})\le 0$.
	Assume, by contradiction, that there exists a sequence $\{u_{n}\} \subset L^{2}(\Omega)$ with $\int_{\Omega}u_{n}^{2}(x)\, \mathrm{d}x=1$ such that
	\[
	\lim\limits_{n\to +\infty} \left( \frac{d_{I}}{2} \int_{\Omega} \int_{\Omega} J(x - y)\left[ u_{n}(y) - u_{n}(x)\right] ^2 \, \mathrm{d}y\mathrm{d}x+\int_{\Omega}\gamma(x)u_{n}^{2}(x)\, \mathrm{d}x \right) = 0.
	\]
	Denote
	$\tilde u_{n} := \frac{1}{\left| \Omega \right| } \int_{\Omega}u_{n}(x) \, \mathrm{d}x$.
	Since $\left|\tilde u_{n} \right| \le \frac{1}{\sqrt{\left| \Omega \right|} }$,
	there exists a subsequence, still denoted by  $\tilde u_{n}$ such that $\tilde u_{n} \to a $ for some $a\in \mathbb{R}$.
	By \eqref{eq2.4}, we have
	\[
	\frac{d_{I}}{2} \int_{\Omega} \int_{\Omega} J(x - y)\left[ u_{n}(y) - u_{n}(x)\right] ^2 \, \mathrm{d}y\mathrm{d}x \ge d_{I}\alpha \int_{\Omega} \left( u_{n}(x)-\tilde u_{n}\right) ^{2}\, \mathrm{d}x.
	\]
	It follows that $\int_{\Omega} \left( u_{n}(x)-a\right)^{2} \, \mathrm{d}x\to 0$. Then, $a^{2}=\frac{1}{\left| \Omega \right| }$, and  we have
	\[
	\lim\limits_{n\to +\infty}\int_{\Omega}\gamma(x)u_{n}^{2}(x)\, \mathrm{d}x =  a^{2} \int_{\Omega}\gamma(x)\, \mathrm{d}x>0,
	\]
which yields a contradiction.
\end{proof}

Applying Proposition \ref{pr5.1} and Theorem \ref{th5.1}, we have the following conclusions.

\begin{proposition}\label{th6.1}
	
	The following statements are valid:
	
	\begin{itemize}
		\item [\rm (i)] $\mathcal{R}_{0}-1$ has the same sign as $s(\mathcal{\hat L}_{1})$, and $ \mathcal{R}_{0}$ is the unique positive solution of $s(\mathcal{ \hat L}_{\mu})=0$;
		
		\item [\rm (ii)]
		\[
		\begin{aligned}
			\mathcal{R}_{0}= & \sup_{\substack{\varphi \in L^2(\Omega) \\ \varphi \neq 0}} \frac{\int_{\Omega} \frac{K\beta(x)}{m(x) + K}\varphi^2(x)\, \mathrm{d}x}{\frac{d_{I}}{2} \int_{\Omega} \int_{\Omega} J(x - y)\left[ \varphi(y) - \varphi(x)\right] ^2 \, \mathrm{d}y\mathrm{d}x + \int_{\Omega} \gamma(x)\varphi^2(x)\, \mathrm{d}x}\\
			=&\inf_{\phi \in \mathcal{Y} } \sup_{x \in \Omega } \dfrac{\frac{K\beta(x)}{m(x) + K}\phi(x)}{-d_I \! \int_{\Omega} J(x- y) \phi(y) \,\mathrm{d}y \! +\! d_{I}\!\int_{\Omega}J(x-y)\,\mathrm{d}y\phi(x)\! + \!\gamma(x)\phi(x)}
			\\[0.4em]
			=& \sup_{\phi \in \mathrm{Int}(Y_{+}) } \inf_{x \in \Omega_{\phi}} \dfrac{\frac{K\beta(x)}{m(x) + K}\phi(x)}{-d_I \! \int_{\Omega} J(x- y) \phi(y) \,\mathrm{d}y \! + \! d_{I}\!\int_{\Omega}J(x-y)\,\mathrm{d}y\phi(x) \! + \! \gamma(x)\phi(x)},
		\end{aligned}
		\]
		where $\mathcal{Y}:= \left\lbrace \phi \in \mathrm{Int}(Y_{+}): \left[ \mathcal{\hat L}_{\infty}\phi\right](x)< 0, \forall x\in \overline{\Omega}  \right\rbrace $, and $\Omega_{\phi}:= \left\lbrace x\in \overline\Omega: \left[ \mathcal{\hat L}_{\infty}\phi\right](x)< 0 \right\rbrace $ for  each $\phi \in \mathrm{Int}(Y_{+}) $.
	\end{itemize}
	
\end{proposition}

\begin{theorem}\label{th6.2}
	The following statements hold:
	
	\begin{itemize}
		\item [\rm (i)] $ \mathcal{R}_{0}$ is nonincreasing in $d_{I}$, and
		\[
		\lim\limits_{d_{I}\to 0} \mathcal{R}_{0}=  \sup_{x \in {\Omega}\setminus\left\lbrace x: \gamma(x)=0 \right\rbrace } \frac{K \beta(x)}{\gamma(x)(m(x)+K)}; \quad  \lim\limits_{d_{I}\to +\infty} \mathcal{R}_{0}= \frac{ \int_{\Omega} \frac{K\beta(x)}{m(x)+K} \, \mathrm{d}x}{\int_{\Omega}\gamma(x)\, \mathrm{d}x}.
		\]

		\item [\rm (ii)] $ \mathcal{R}_{0}$ is nondecreasing in $K$.  Moreover,
		\[
		\lim\limits_{K\to +\infty} \mathcal{R}_{0} \! = \! \sup_{\substack{\varphi \in L^2(\Omega) \\ \varphi \neq 0}} \frac{\int_{\Omega}\beta(x)\varphi^2(x)\, \mathrm{d}x}{\frac{d_{I}}{2} \int_{\Omega} \int_{\Omega} J(x - y)\left[ \varphi(y)\! - \! \varphi(x)\right] ^2 \, \mathrm{d}y\mathrm{d}x \! + \! \int_{\Omega} \gamma(x)\varphi^2(x)\, \mathrm{d}x};
		\]
		and, if {\rm(\textbf{A3})} holds, then $ \mathcal{R}_{0}$ is strictly increasing in $K$, and
		\[
		\lim\limits_{K\to 0} \mathcal{R}_{0} = 0.
		\]
		If {\rm(\textbf{A3}$'$)} holds, then
		\[
		\max\limits_{x\in \varSigma} \frac{\beta(x)}{d_{I}\int_{\Omega}J(x-y) \, \mathrm{d}y +  \gamma(x)} \le \lim_{K \to 0} \mathcal{R}_{0} \le  \max\limits_{x\in \varSigma}\frac{\beta(x)}{-s(\mathcal{\hat L}_{\infty})};
		\]
		moreover, if $\varSigma $  is a finite set, then
		\[
		\lim\limits_{K\to 0} \mathcal{R}_{0} =  \max\limits_{x\in \varSigma} \frac{\beta(x)}{d_{I}\int_{\Omega}J(x-y) \, \mathrm{d}y +  \gamma(x)}.
		\]
	\end{itemize}

\end{theorem}

\begin{proof}
	
	(i)
	According to Proposition \ref{th6.1}, it is immediate that $\mathcal{R}_{0}$ is nonincreasing in $d_{I}$.
	Set
	$$
	\varUpsilon:=\left\lbrace  \phi \in L^{2}(\Omega) : \int_{\Omega}\gamma(x)\phi^{2}(x) \, \mathrm{d}x =0\right\rbrace.
	$$
	It is easy to show that $\overline{L^{2}(\Omega)\setminus\varUpsilon} = L^{2}(\Omega)$ thanks to assumption (\textbf{A2}). It follows that
	\[
	\begin{aligned}
		& \sup_{\substack{\varphi \in L^2(\Omega) \\ \varphi \neq 0}} \frac{\int_{\Omega} \frac{K\beta(x)}{m(x) + K}\varphi^2(x)\, \mathrm{d}x}{\frac{d_{I}}{2} \int_{\Omega} \int_{\Omega} J(x - y)\left[ \varphi(y) - \varphi(x)\right] ^2 \, \mathrm{d}y\mathrm{d}x + \int_{\Omega} \gamma(x)\varphi^2(x)\, \mathrm{d}x}\\
		= & \sup_{\substack{\varphi \in L^2(\Omega)\setminus\varUpsilon}} \frac{\int_{\Omega} \frac{K\beta(x)}{m(x) + K}\varphi^2(x)\, \mathrm{d}x}{\frac{d_{I}}{2} \int_{\Omega} \int_{\Omega} J(x - y)\left[ \varphi(y) - \varphi(x)\right] ^2 \, \mathrm{d}y\mathrm{d}x + \int_{\Omega} \gamma(x)\varphi^2(x)\, \mathrm{d}x}.
	\end{aligned}
	\]
	One can show that
	\[
	\begin{aligned}
		& \lim\limits_{d_{I}\to 0}\sup_{\substack{\varphi \in L^2(\Omega)\setminus\varUpsilon}} \frac{\int_{\Omega} \frac{K\beta(x)}{m(x) + K}\varphi^2(x)\, \mathrm{d}x}{\frac{d_{I}}{2} \int_{\Omega} \int_{\Omega} J(x - y)\left[ \varphi(y) - \varphi(x)\right] ^2 \, \mathrm{d}y\mathrm{d}x + \int_{\Omega} \gamma(x)\varphi^2(x)\, \mathrm{d}x} \\
		= &\sup_{\substack{\varphi \in L^2(\Omega)\setminus\varUpsilon}}\frac{\int_{\Omega} \frac{K\beta(x)}{m(x) + K}\varphi^2(x)\, \mathrm{d}x}{\int_{\Omega} \gamma(x)\varphi^2(x)\, \mathrm{d}x}.
	\end{aligned}
	\]
	Next, we  show that
	\[
	\sup_{\substack{\varphi \in L^2(\Omega)\setminus\varUpsilon}}\frac{\int_{\Omega} \frac{K\beta(x)}{m(x) + K}\varphi^2(x)\, \mathrm{d}x}{\int_{\Omega} \gamma(x)\varphi^2(x)\, \mathrm{d}x}= F: =\sup_{x \in {\Omega}\setminus\left\lbrace x: \gamma(x)=0 \right\rbrace } \frac{K \beta(x)}{\gamma(x)(m(x)+K)}.
	\]
	For every $x_{0}\in  {\Omega}\setminus\left\lbrace x: \gamma(x)=0 \right\rbrace$, define
	\[
	\psi_{n}(x)= \begin{cases}
		\frac{1}{\sqrt{\left|B_{\frac{1}{n}}(x_{0})\cap \Omega \right| }}, & x\in B_{\frac{1}{n}}(x_{0})\cap\Omega,\\
		0, & x \in {\Omega}\setminus B_{\frac{1}{n}}(x_{0}).
	\end{cases}
	\]
	By continuity, for sufficiently large $n$ we have $ \psi_{n}\in L^2(\Omega)\setminus\varUpsilon$, and
	\[
	\lim\limits_{n\to +\infty} \dfrac{ \int_{\Omega}\frac{K\beta(x)}{m(x) + K}\psi_{n}^2(x)\, \mathrm{d}x}{\int_{\Omega} \gamma(x)\psi_{n}^2(x)\, \mathrm{d}x}= \frac{K \beta(x_{0})}{\gamma(x_{0})(m(x_{0})+K)}.
	\]
	This implies
	$$
	\sup_{\substack{\varphi \in L^2(\Omega)\setminus\varUpsilon}}\frac{\int_{\Omega} \frac{K\beta(x)}{m(x) + K}\varphi^2(x)\, \mathrm{d}x}{\int_{\Omega} \gamma(x)\varphi^2(x)\, \mathrm{d}x}\ge F.
	$$
	
	If $F=+\infty$, the conclusion is immediate. Consider the case $F<+\infty$.
	Then, whenever $\gamma(x_0)=0$, we must have $\beta(x_0)=0$. Otherwise the supremum would be $+\infty$. Consequently, $\frac{K\beta(x)}{m(x) + K}\le F \gamma(x) $ for all $x\in \Omega$. Hence
	$$
	\sup\limits_{\substack{\varphi \in L^2(\Omega)\setminus\varUpsilon}}\frac{\int_{\Omega} \frac{K\beta(x)}{m(x) + K}\varphi^2(x)\, \mathrm{d}x}{\int_{\Omega} \gamma(x)\varphi^2(x)\, \mathrm{d}x}\le F.
	$$
	Thus, the desired equality follows.
	
	We denote
	$$
	\tilde F:= \frac{ \int_{\Omega} \frac{K\beta(x)}{m(x)+K} \, \mathrm{d}x}{\int_{\Omega}\gamma(x)\, \mathrm{d}x}.
	$$
	It is straightforward to verify that
	$
	\mathcal{R}_{0}\ge \tilde F.
	$
	Let  $\varepsilon>0$ be arbitrarily small.
	For  each $d_{I}$, there exists $\phi_{d_{I}}\in L^{2}(\Omega)$ with $\left\| \phi_{d_{I}}\right\| _{L^{2}}=1$ such that
	\[
	\frac{\int_{\Omega} \frac{K\beta(x)}{m(x) + K}\phi_{d_{I}}^2(x)\, \mathrm{d}x}{\frac{d_{I}}{2} \int_{\Omega} \int_{\Omega} J(x - y)\left[ \phi_{d_{I}}(y) - \phi_{d_{I}}(x)\right] ^2 \, \mathrm{d}y\mathrm{d}x + \int_{\Omega} \gamma(x)\phi_{d_{I}}^2(x)\, \mathrm{d}x} \ge\mathcal{R}_{0} -\varepsilon \ge \tilde F -\varepsilon>0.
	\]
	Recall that
	\[
	\frac{d_{I}}{2} \int_{\Omega} \int_{\Omega} J(x - y)\left[\phi_{d_{I}}(y) - \phi_{d_{I}}(x)\right] ^2 \, \mathrm{d}y\mathrm{d}x \ge d_{I}\alpha \int_{\Omega} \left( \phi_{d_{I}}(x)-\tilde \phi_{d_{I}}\right) ^{2}\, \mathrm{d}x,
	\]
	where $\tilde \phi_{d_{I}}:= \frac{1}{\left| \Omega\right| } \int_{\Omega}\phi_{d_{I}}(x) \, \mathrm{d}x.$
	Hence,
	\[
	\int_{\Omega} \left( \phi_{d_{I}}(x)-\tilde \phi_{d_{I}}\right) ^{2}\, \mathrm{d}x\le \dfrac{1}{d_{I} \alpha}\left(  \frac{\int_{\Omega} \frac{K\beta(x)}{m(x) + K}\phi_{d_{I}}^2(x)\, \mathrm{d}x}{\tilde F -\varepsilon} -  \int_{\Omega} \gamma(x)\phi_{d_{I}}^2(x)\, \mathrm{d}x\right).
	\]
	Note that, passing to a subsequence if necessary (still denoted by  $\tilde \phi_{d_{I}}$), we have $\tilde \phi_{d_{I}} \to a $ with $a \in \mathbb{R}$. Hence, $\int_{\Omega} \left( \phi_{d_{I}}(x)-a\right) ^{2}\, \mathrm{d}x \to 0$, and $a^{2}=\frac{1}{\left| \Omega\right| }$. Thus,
	\[
	\limsup\limits_{d_{I}\to +\infty} \frac{\int_{\Omega} \frac{K\beta(x)}{m(x) + K}\phi_{d_{I}}^2(x)\, \mathrm{d}x}{\frac{d_{I}}{2} \int_{\Omega} \int_{\Omega} J(x - y)\left[ \phi_{d_{I}}(y) - \phi_{d_{I}}(x)\right] ^2 \, \mathrm{d}y\mathrm{d}x + \int_{\Omega} \gamma(x)\phi_{d_{I}}^2(x)\, \mathrm{d}x} \le  \tilde F.
	\]
	Therefore,
	$$
	\tilde F \le \liminf_{d_{I}\to +\infty} \mathcal{R}_{0} \le  \limsup_{d_{I}\to +\infty} \mathcal{R}_{0} \le \tilde F +\varepsilon.
	$$
	The arbitrariness of $\varepsilon$ completes the proof.

	(ii) According to Proposition \ref{th6.1}, it is immediate that $\mathcal{R}_{0}$ is nondecreasing in $K$.
	It is straightforward to verify that
	$$
	\lim\limits_{K\to +\infty}  \frac{K}{m(x) + K} =1 \text{ uniformly in } x \in \overline{\Omega}.
	$$
	For any $\varepsilon>0$, there exists $N$ such that for all $K\ge N$,
	\[
	\left| \frac{\int_{\Omega} \left( \frac{K}{m(x) + K} -1 \right)\beta(x) \phi^2(x)\, \mathrm{d}x}{\frac{d_{I}}{2} \int_{\Omega} \int_{\Omega} J(x - y)\left[ \phi(y) - \phi(x)\right] ^2 \, \mathrm{d}y\mathrm{d}x + \int_{\Omega} \gamma(x)\phi^2(x)\, \mathrm{d}x}   \right| \le \dfrac{\varepsilon}{- s(\mathcal{\hat L}_{\infty})}, \  \forall \phi\in L^{2}(\Omega).
	\]
	Therefore,
	\[
	\lim\limits_{K\to +\infty} \mathcal{R}_{0} =  \sup_{\substack{\varphi \in L^2(\Omega) \\ \varphi \neq 0}} \frac{\int_{\Omega}\beta(x)\varphi^2(x)\, \mathrm{d}x}{\frac{d_{I}}{2} \int_{\Omega} \int_{\Omega} J(x - y)\left[ \varphi(y) - \varphi(x)\right] ^2 \, \mathrm{d}y\mathrm{d}x + \int_{\Omega} \gamma(x)\varphi^2(x)\, \mathrm{d}x}.
	\]

	If (\textbf{A3}) holds, we can deduce that   $ \mathcal{R}_{0}$ is strictly increasing in $K$, and
	$$
	\lim\limits_{K\to 0}  \frac{K}{m(x) + K} =0 \text{ uniformly in } x \in \overline{\Omega}.
	$$
	Similarly, we can prove that  $\lim\limits_{K\to 0} \mathcal{R}_{0}=0 $.

	Next, we consider the case where (\textbf{A3$'$}) holds.
	By Proposition \ref{th6.1}(i), we have
	\[
	\max\limits_{x \in \overline{\Omega}}\left\lbrace  -d_{I} \int_{\Omega} J(x-y)\, \mathrm{d}y  -\gamma(x)  +\frac{K\beta(x)}{\mathcal{R}_{0} \left( m(x) + K\right) } \right\rbrace \le s(\mathcal{\hat L}_{\mathcal{R}_{0}})=0.
	\]
	Consequently,
	\begin{equation}\label{eq6.3}
		\mathcal{R}_{0}\ge  \max\limits_{x \in \overline{\Omega}} \frac{K\beta(x)}{(m(x)+K)(d_{I}\int_{\Omega}J(x-y)\, \mathrm{d}y+\gamma(x))}\ge  \max\limits_{x\in \varSigma} \frac{\beta(x)}{d_{I}\int_{\Omega}J(x-y)\, \mathrm{d}y+\gamma(x)}.
	\end{equation}
	For any $\varepsilon>0$, define $B_{\varepsilon}(\varSigma):=\left\lbrace x\in \overline{\Omega}: \operatorname{dist}(x,\varSigma) \le  \varepsilon \right\rbrace $. Then, for $K$ sufficiently small,
	\[
	\left| \int_{\Omega} \frac{K\beta(x)}{m(x) + K}\phi^2(x)\, \mathrm{d}x\right| \le \varepsilon \left\|  \beta\right\|_{Y}+  \max\limits_{x\in  B_{\varepsilon}(\varSigma)} \beta(x), \forall \phi\in L^{2}(\Omega) \text{ with } \left\| \phi\right\|_{L^{2}}=1.
	\]
	Hence
	$$
	\lim_{K\to 0}\mathcal{R}_{0} \le \frac{\varepsilon \left\|  \beta\right\|_{Y}+  \max\limits_{x\in  B_{\varepsilon}(\varSigma)} \beta(x)}{-s(\mathcal{\hat L}_{\infty})}.
	$$
	Letting $\varepsilon\to 0$, we arrive at
	\[
	\lim_{K \to 0}\mathcal{R}_{0} \le  \max\limits_{x\in \varSigma}\frac{\beta(x)}{-s(\mathcal{\hat L}_{\infty})}.
	\]

	Next, we first consider the special case where $\varSigma:=\left\lbrace x_{0} \right\rbrace $.
	If $\beta(x_{0})=0$,  the desired conclusion follows directly from the previous discussion. Thus, we assume $ \beta(x_{0})>0$.

	For any $\varepsilon>0$, define  $B_{\varepsilon}(x_{0}):= \left\lbrace x\in \overline{\Omega}: \left\| x-x_{0}\right\|\le  \varepsilon \right\rbrace  $. Then there exists $N=N(\varepsilon)$ such that
	$$
	\left| \frac{K}{m(x) + K}\right| \le \varepsilon \text{  for all } x\in \overline{\Omega}\setminus B_{\varepsilon}(x_{0}) \text{  provided } K\le N.
	$$
	For each $K$, there exists $\phi_{K}\in L^{2}(\Omega)$ with $\left\| \phi_{K}\right\| _{L^{2}(\Omega)}=1$ such that
	\[
	\frac{\int_{\Omega} \frac{K\beta(x)}{m(x) + K}\phi_{K}^2(x)\, \mathrm{d}x}{\frac{d_{I}}{2} \int_{\Omega} \int_{\Omega} J(x - y)\left[ \phi_{K}(y) - \phi_{K}(x)\right] ^2 \, \mathrm{d}y\mathrm{d}x + \int_{\Omega} \gamma(x)\phi_{K}^2(x)\, \mathrm{d}x} \ge\mathcal{R}_{0} -\varepsilon.
	\]
	Write $\phi_{K}=\phi_{K,\varepsilon}+\phi_{K,\varepsilon}^{c}$, where $ \phi_{K,\varepsilon}(x)=\phi_{K}(x)$ for $x\in B_{\varepsilon}(x_{0})$ and $\phi_{K,\varepsilon}(x)=0 $  for $x\in \Omega\setminus B_{\varepsilon}(x_{0})$. Set
	$$
	\eta_{K,\varepsilon}=\int_{B_{\varepsilon}(x_{0})} \phi_{K,\varepsilon}^{2}(x)\, \mathrm{d}x \in [0,1].
	$$
	By straightforward calculation, we obtain the following estimates:
	\[
	\begin{aligned}
		\left| \int_{\Omega} \frac{K\beta(x)}{m(x) + K}\phi_{K}^2(x)\, \mathrm{d}x\right| \! \le & \varepsilon\int_{\Omega\setminus B_{\varepsilon}(x_{0})}\! \left( \phi_{K,\varepsilon}^{c}(x)\right)^{2}\, \mathrm{d}x \! + \! \max\limits_{x\in  B_{\varepsilon}(x_{0})} \beta(x) \int_{B_{\varepsilon}(x_{0})} \phi_{K,\varepsilon}^{2}(x)\, \mathrm{d}x\\[0.4em]
		=& \varepsilon (1- \eta_{K,\varepsilon}) + \eta_{K,\varepsilon}  \max\limits_{x\in  B_{\varepsilon}(x_{0})}\beta(x),
	\end{aligned}
	\]
	and
	\[
	\begin{aligned}
		& \frac{d_{I}}{2} \int_{\Omega} \int_{\Omega} J(x - y)\left[ \phi_{K}(y) - \phi_{K}(x)\right] ^2 \, \mathrm{d}y\mathrm{d}x + \int_{\Omega} \gamma(x)\phi_{K}^2(x)\, \mathrm{d}x\\
		= & d_{I}\int_{\Omega\setminus B_{\varepsilon}(x_{0})}\int_{\Omega}J(x-y) \, \mathrm{d}y\left( \phi_{K,\varepsilon}^{c}(x)\right)^{2}\, \mathrm{d}x + \int_{\Omega\setminus B_{\varepsilon}(x_{0})} \gamma(x) \left( \phi_{K,\varepsilon}^{c}(x)\right)^{2}\, \mathrm{d}x\\
		& + d_{I}\int_{B_{\varepsilon}(x_{0})}\int_{\Omega}J(x-y) \, \mathrm{d}y \phi_{K,\varepsilon}^{2}(x)\, \mathrm{d}x + \int_{B_{\varepsilon}(x_{0})} \gamma(x)  \phi_{K,\varepsilon}^{2}(x)\, \mathrm{d}x\\
		& - d_{I} \int_{\Omega\setminus B_{\varepsilon}(x_{0})}\int_{\Omega\setminus B_{\varepsilon}(x_{0})}J(x-y)  \phi_{K,\varepsilon}^{c}(y)\phi_{K,\varepsilon}^{c} (x) \, \mathrm{d}y \mathrm{d}x \\
		&- d_{I} \int_{B_{\varepsilon}(x_{0})}\int_{ B_{\varepsilon}(x_{0})}J(x-y)  \phi_{K,\varepsilon}(y)\phi_{K,\varepsilon}(x) \, \mathrm{d}y \mathrm{d}x \\
		& - 2d_{I} \int_{B_{\varepsilon}(x_{0})}\int_{\Omega} J(x-y)\phi_{K,\varepsilon}^{c}(y)\phi_{K,\varepsilon}(x) \, \mathrm{d}y \mathrm{d}x\\
		\ge & -s(\mathcal{\hat L}_{\infty})(1-\eta_{K,\varepsilon})+ \eta_{K,\varepsilon}\left( d_{I} \min\limits_{x\in B_{\varepsilon}(x_{0})}\int_{\Omega}J(x-y) \, \mathrm{d}y +  \min\limits_{x\in B_{\varepsilon}(x_{0})} \gamma(x) \right) \\
		& - d_{I} \left\| J\right\| _{L^{\infty}}\left| B_{\varepsilon}(x_{0})\right| -2d_{I} \left\| J\right\| _{L^{\infty}}\left| \Omega\right|^{\frac{1}{2}} \left| B_{\varepsilon}(x_{0})\right|^{\frac{1}{2}}.
	\end{aligned}
	\]
	Denote
	$$
	\kappa(\varepsilon):=  d_{I} \left\| J\right\| _{L^{\infty}}\left| B_{\varepsilon}(x_{0})\right| +2d_{I} \left\| J\right\| _{L^{\infty}}\left| \Omega\right|^{\frac{1}{2}} \left| B_{\varepsilon}(x_{0})\right|^{\frac{1}{2}}.
	$$
	Since $-s(\mathcal{\hat L}_{\infty})>0$, $\lim\limits_{\varepsilon\to 0}\kappa(\varepsilon)=0$, and
	\[
	\lim_{\varepsilon\to 0} \left(  \min\limits_{x\in B_{\varepsilon}(x_{0})} d_{I}\int_{\Omega}J(x-y) \, \mathrm{d}y +  \min\limits_{x\in B_{\varepsilon}(x_{0})} \gamma(x) \right) = d_{I}\int_{\Omega}J(x_{0}-y) \, \mathrm{d}y +  \gamma(x_{0})>0,
	\]
	there exists $\varepsilon_{0}$, such that for $\varepsilon\le\varepsilon_{0} $,
	\[
	-s(\mathcal{\hat L}_{\infty})-\kappa(\varepsilon)>0, \text{ and }
	\left(  \min\limits_{x\in B_{\varepsilon}(x_{0})} d_{I}\int_{\Omega}J(x-y) \, \mathrm{d}y +  \min\limits_{x\in B_{\varepsilon}(x_{0})} \gamma(x) \right) -\kappa(\varepsilon)>0.
	\]
	Note that we can write
	\[
	\kappa(\varepsilon)=\kappa(\varepsilon)(1-\eta_{K,\varepsilon}) + \kappa(\varepsilon) \eta_{K,\varepsilon}.
	\]
	Using the  inequality $\frac{a(1-\eta)+c\eta}{b(1-\eta)+d\eta} \le \max\{\frac{a}{b},\frac{c}{d}\}$ for $a,b,c,d>0$ and $\eta\in[0,1]$, we obtain
	\[
	\begin{aligned}
		& \frac{\int_{\Omega} \frac{K\beta(x)}{m(x) + K}\phi_{K}^2(x)\, \mathrm{d}x}{\frac{d_{I}}{2} \int_{\Omega} \int_{\Omega} J(x - y)\left[ \phi_{K}(y) - \phi_{K}(x)\right] ^2 \, \mathrm{d}y\mathrm{d}x + \int_{\Omega} \gamma(x)\phi_{K}^2(x)\, \mathrm{d}x} \\[0.4em]
		\le & \dfrac{\varepsilon (1- \eta_{K,\varepsilon}) + \eta_{K,\varepsilon}  \max\limits_{x\in  B_{\varepsilon}(x_{0})}\beta(x)}{
			\left( -s(\mathcal{\hat L}_{\infty})-\kappa(\varepsilon)\right) (1-\eta_{K,\varepsilon})  + \eta_{K,\varepsilon}\Bigl( d_{I} \min\limits_{x\in B_{\varepsilon}(x_{0})}\int_{\Omega}J(x-y) \, \mathrm{d}y +  \min\limits_{x\in B_{\varepsilon}(x_{0})} \gamma(x)   - \kappa(\varepsilon)\Bigr)
		}\\[0.4em]
		\le & \max\left\lbrace   \frac{\varepsilon}{
			-s(\mathcal{\hat L}_{\infty}) - \kappa(\varepsilon)
		}, \frac{ \max\limits_{x\in  B_{\varepsilon}(x_{0})}\beta(x)}{
			\Bigl( d_{I} \min\limits_{x\in B_{\varepsilon}(x_{0})}\int_{\Omega}J(x-y) \, \mathrm{d}y +  \min\limits_{x\in B_{\varepsilon}(x_{0})} \gamma(x) \Bigr) -
			\kappa(\varepsilon)
		} \right\rbrace.
	\end{aligned}
	\]
	Now, choose \(\varepsilon_1\le\varepsilon_0\) such that for all \(\varepsilon\le\varepsilon_1\),
	\[
	\frac{\varepsilon}{
		-s(\mathcal{\hat L}_{\infty}) - \kappa(\varepsilon)
	} \le  \frac{ \max\limits_{x\in B_{\varepsilon}(x_{0})}\beta(x)}{
		\Bigl( d_{I} \min\limits_{x\in B_{\varepsilon}(x_{0})}\int_{\Omega}J(x-y) \, \mathrm{d}y +  \min\limits_{x\in B_{\varepsilon}(x_{0})} \gamma(x) \Bigr)
		-  \kappa(\varepsilon)
	}.
	\]
	Then, for such \(\varepsilon\) and for all \(K\le N(\varepsilon)\), we have
	\[
	\mathcal{R}_0 - \varepsilon \le \frac{ \max\limits_{x\in{B}_\varepsilon(x_0)}\beta(x)}{ d_I \min\limits_{x\in B_\varepsilon(x_0)}\int_\Omega J(x-y)\,dy +  \min\limits_{x\in B_\varepsilon(x_0)}\gamma(x) - \kappa(\varepsilon) },
	\]
	and consequently
	\[
	\limsup_{K\to0}\mathcal{R}_0 \le \frac{ \max\limits_{x\in{B}_\varepsilon(x_0)}\beta(x)}{ d_I \min\limits_{x\in B_\varepsilon(x_0)}\int_\Omega J(x-y)\,dy +  \min\limits_{x\in B_\varepsilon(x_0)}\gamma(x) - \kappa(\varepsilon) } + \varepsilon.
	\]
	Letting \(\varepsilon\to0\) and using the continuity of the involved functions, we obtain
	\[
	\limsup_{K\to 0} \mathcal{R}_{0}\le \frac{\beta(x_{0})}{d_{I}\int_{\Omega}J(x_{0}-y) \, \mathrm{d}y +  \gamma(x_{0})}.
	\]
	By a similar argument, this estimate extends to the case where \(\Sigma\) is a finite set:
	\begin{equation}\label{eq6.5}
		\limsup_{K\to 0} \mathcal{R}_{0}\le  \max\limits_{x\in \varSigma} \frac{\beta(x)}{d_{I}\int_{\Omega}J(x-y) \, \mathrm{d}y +  \gamma(x)}.
	\end{equation}
	Combining with \eqref{eq6.3} and \eqref{eq6.5}, we obtain
	$$
	\lim\limits_{K\to 0} \mathcal{R}_{0}=  \max\limits_{x\in \varSigma} \frac{\beta(x)}{d_{I}\int_{\Omega}J(x-y) \, \mathrm{d}y +  \gamma(x)}.
	$$
\end{proof}

We note that when $\beta$, $\gamma$, and $m$ are all positive functions, the conclusions of Theorem \ref{th6.2} were previously established by Feng et al.~\cite[Theorem 2.1]{Feng_Li_Yang_2024}. Their proof relies on the theory developed by Zhang and Zhao~\cite{MR4348165} and characterizes the limit of the basic reproduction ratio through the asymptotic profiles of the corresponding principal eigenvalues. In particular, as $K\to\infty$, the limit is given by the spectral radius of a related operator, in contrast to the variational characterizations derived earlier in this paper.

Here, we further extend these conclusions to the degenerate case where the coefficient functions are allowed to vanish. For the first statement in Theorem \ref{th6.2}(i), the limit may become infinite when $\gamma(x)$ possesses zeros, representing a critical state. Another particularly noteworthy scenario occurs when both $\beta$ and $\gamma$ vanish at the same point $x_0$; in this case, the limiting value is determined by $\limsup\limits_{x \to x_0} \frac{\beta(x)}{\gamma(x)}$.

It is important to note that in the degenerate case, the limiting behavior as $K \to 0$ differs fundamentally between nonlocal and local diffusion. For local diffusion, it follows from Gao et al.~\cite{MR4687008} that if $\varSigma$ is a null set, the limit of the basic reproduction ratio is zero. This conclusion, however, fails for nonlocal dispersal. In fact, if $\beta$ is positive at any point in $\varSigma$, the basic reproduction ratio admits a positive lower bound. This discrepancy is rooted in the different variational structures of local and nonlocal dispersal operators. In particular, the local diffusion operator involves a gradient energy. Consequently, its spectrum is not necessarily bounded below by the product operator. In contrast, the nonlocal dispersal operator retains a nontrivial integral contribution even when the corresponding test functions concentrate near isolated points, which  yields a positive lower bound.

Furthermore, the compactness properties of local diffusion guarantee the existence of a principal eigenvalue, and the corresponding eigenfunctions admit a convergent subsequence. This precludes the possibility that the mass of the eigenfunctions concentrates at degenerate points in the limit, thereby facilitating a straightforward characterization of the limit. In contrast, the proof of Theorem \ref{th6.2} for the case where $\varSigma$ is a finite set reveals that  such a concentration phenomena can indeed occur in the nonlocal dispersal case, leading to notably different conclusions.

\section*{Availability of data and material}
Not applicable.

\section*{Competing interests}
The authors declare that they have no competing interests.

\section*{Acknowledgments}
This work  was supported  by the National Natural Science Foundation of China (Nos. 12471176 and 12071491) and Guangdong Basic and Applied Basic Research Foundation (No. 2025A1515012221).  X. Lin was additionally supported by National Natural Science Foundation of China (No. 12601329) and China Postdoctoral Science Foundation (No. 2025M783159).


\end{document}